\documentclass[a4paper, reqno]{amsart} %
\usepackage{amssymb, mathdots}
\usepackage{mathabx}
\usepackage{mathrsfs, hhline}
\usepackage[utf8]{inputenc}
\usepackage{amsmath,  graphicx, wasysym}
\usepackage{hyperref}
\usepackage{caption}
\usepackage{mathtools, tensor}
\usepackage{tikz}
\usetikzlibrary{matrix, arrows.meta}
\usetikzlibrary{cd}
\usepackage{wrapfig}  
\usepackage{enumerate}
\DeclareSymbolFont{bbold}{U}{bbold}{m}{n}
\DeclareSymbolFontAlphabet{\mathbbold}{bbold}
\def\qmod#1#2{{\hbox{}^{\displaystyle{#1}}}\!\big/\!\hbox{}_{
\displaystyle{#2}}}

\def\C{{\mathbb C}}

\def\H{{\mathbb H}}

\def\N{{\mathbb N}}

\def\R{{\mathbb R}}

\def\textmap#1{\mathop{\vbox{\ialign{
                                  ##\crcr
      ${\scriptstyle\hfil\;\;#1\;\;\hfil}$\crcr
      \noalign{\kern 1pt\nointerlineskip}
      \rightarrowfill\crcr}}\;}}
\def\bigtextmap#1{\mathop{\vbox{\ialign{
                                  ##\crcr
      ${\hfil\;\;#1\;\;\hfil}$\crcr
      \noalign{\kern 1pt\nointerlineskip}
      \rightarrowfill\crcr}}\;}}
      
\newcommand{\cal}{\mathcal}
\def\textlmap#1{\mathop{\vbox{\ialign{
                                  ##\crcr
      ${\scriptstyle\hfil\;\;#1\;\;\hfil}$\crcr
      \noalign{\kern-1pt\nointerlineskip}
      \leftarrowfill\crcr}}\;}}
 
\def\ag{{\mathfrak a}}

\def\g{{\mathfrak g}}
\def\hg{{\mathfrak h}}

\def\kg{{\mathfrak k}}

\def\qg{{\mathfrak q}}

\def\Hg{{\mathfrak H}}

\theoremstyle{remark}
\newtheorem{ex}{Example}[section]

\theoremstyle{plain}
\newtheorem{sz}{Satz}[section]
\newtheorem{thry}[sz]{Theorem}
\newtheorem{pr}[sz]{Proposition}
\newtheorem{re}[sz]{Remark}
\newtheorem{co}[sz]{Corollary}
\newtheorem{dt}[sz]{Definition}
\newtheorem{lm}[sz]{Lemma}

\def\Tr{\mathrm {Tr}}
\def\End{\mathrm {End}}
\def\Aut{\mathrm {Aut}}

\def\U{\mathrm{U}}
\def\SU{\mathrm {SU}}
\def\SO{\mathrm {SO}}

\def\GL{\mathrm {GL}}
\def\PGL{\mathrm {PGL}}
\def\PSL{\mathrm {PSL}}
\def\SL{\mathrm {SL}}
\def\sl{\mathrm {sl}}
 
\def\u{\mathrm {u}}
\def\su{\mathrm {su}}
\def\so{\mathrm {so}}

\def\gl{\mathrm {gl}}

\def\Hom{\mathrm{Hom}}
\def\Herm{\mathrm{Herm}}

\def\vol{\mathrm{vol}}

\def\id{ \mathrm{id}}
\def\im{\mathrm{im}}
\def\rk{\mathrm {rk}}
\def\ad{\mathrm {ad}}
\def\Ad{\mathrm {Ad}}

\def\S{\mathrm{S}}

\newcommand\smvee{{\hskip -0.15ex \raise 0.2ex\hbox{$\scriptscriptstyle\vee$}\hskip -0.3ex}}
\def\trp#1{\tensor[^{\mathrm{t}}]{#1}{}}

\makeatletter
\DeclareFontFamily{OMX}{MnSymbolE}{}
\DeclareSymbolFont{MnLargeSymbols}{OMX}{MnSymbolE}{m}{n}
\SetSymbolFont{MnLargeSymbols}{bold}{OMX}{MnSymbolE}{b}{n}
\DeclareFontShape{OMX}{MnSymbolE}{m}{n}{
    <-6>  MnSymbolE5
   <6-7>  MnSymbolE6
   <7-8>  MnSymbolE7
   <8-9>  MnSymbolE8
   <9-10> MnSymbolE9 
  <10-12> MnSymbolE10
  <12->   MnSymbolE12
}{}
\DeclareFontShape{OMX}{MnSymbolE}{b}{n}{
    <-6>  MnSymbolE-Bold5
   <6-7>  MnSymbolE-Bold6
   <7-8>  MnSymbolE-Bold7
   <8-9>  MnSymbolE-Bold8
   <9-10> MnSymbolE-Bold9
  <10-12> MnSymbolE-Bold10
  <12->   MnSymbolE-Bold12
}{}

\let\llangle\@undefined
\let\rrangle\@undefined
\DeclareMathDelimiter{\llangle}{\mathopen}%
                     {MnLargeSymbols}{'164}{MnLargeSymbols}{'164}
\DeclareMathDelimiter{\rrangle}{\mathclose}%
                     {MnLargeSymbols}{'171}{MnLargeSymbols}{'171}
\makeatother
\newcommand{\vertiii}[1]{{\left\vert\kern-0.25ex\left\vert\kern-0.25ex\left\vert #1 
    \right\vert\kern-0.25ex\right\vert\kern-0.25ex\right\vert}} 

\begin{document} 

\title[A generalisation of Bryant theorem]{A gauge theoretical generalization of Bryant's correspondence}
\author{Andrei Teleman}
\address{Andrei Teleman: Aix Marseille Univ, CNRS,  I2M, UMR 7373, Marseille, France, email: andrei.teleman@univ-amu.fr}

\begin{abstract}

A classical theorem in the theory of minimal surfaces establishes a correspondence between minimal surfaces in $\R^n$ and null holomorphic curves in $\C^n$. A  hyperbolic version of this correspondence is due to Bryant  who showed that null holomorphic curves in $\SL(2,\C)$ correspond to  CMC-1  surfaces   in the hyperbolic space $\H^3$.  

Several authors   obtained and studied a relativistic version of Bryant's correspondence:   CMC-1 immersions in the hyperbolic space $\H^3$ are replaced by space-like  CMC-1 immersion in the de Sitter space $\mathrm{dS}_3$.

We  prove a mutual generalisation  of all these results: let $H$ be a real Lie group, $P\textmap{\pi}M$ a principal $H$-bundle, $A$ a connection on $P$ and $\alpha\in A^1_\Ad(P,\hg)$ a tensorial 1-form of type $\Ad$ which induces isomorphisms $A_\xi\textmap{\simeq} \hg$. Such a pair $(\alpha,A)$ defines an almost complex structure $J^\alpha_A$ on $P$, which is integrable if and only $(\alpha,A)$ is a solution of a  gauge-invariant non-linear first order differential system.  A non-degenerate symmetric $\Ad_H$-invariant bilinear form $g$  on $\hg$ defines pseudo-Riemannian metrics  $g^\alpha_M$, $\g^\alpha_A$ on $M$, respectively $P$, and a non-degenerate bilinear form $\omega^{\alpha,g}_A:T_P\times_P T_P\to \C$ which is $\C$-bilinear with respect to $J^\alpha_A$ and holomorphic when $J^\alpha_A$ is integrable. Assuming that this is the case, we have a Bryant type correspondence between space-like, $\omega^{\alpha,g}_A$-isotropic holomorphic immersions $Y\to P$ and space-like conformal immersions $Y\to (M,g^\alpha_M)$ whose mean curvature vector field is given by a simple explicit formula. 

In particular, one obtains such a correspondence for any principal bundle of the form $G\to  G/H$, where $G$ is a complex Lie group, and $H$ is a real form of $G$ endowed with a non-degenerate, $\Ad_H$-invariant, symmetric bilinear form $g$ on its Lie-algebra $\hg$.

\end{abstract}


\maketitle

\tableofcontents

\section*{Acknowledgements}
I thank my former PhD students Raphael Zentner and Nicolas Al Choueiry for useful remarks and suggestions concerning the presentation of this article.  

\section{Introduction}
\label{intro}

Let  $Y$ be a connected surface, $J$ an almost complex structure on $Y$ and $(Y,J)$ the corresponding Riemann surface.

We start by recalling the following well known fundamental result in the theory of minimal surfaces in the Euclidian space (see \cite{Chern}, \cite{Oss}, \cite{ChOss}, \cite[section III.3]{La}, \cite[Theorem 7.1]{Fo}):
\begin{thry}\label{Th1-intro}  Let $\beta=(\beta_1,\dots,\beta_n)\in \Omega^1(Y,\C^n)$ be a $\C^n$-valued holomorphic 1-form on $Y$ satisfying the conditions: 
\begin{enumerate}

\item[(C1)] $\beta$  has no real periods on $Y$, i.e. for any smooth loop $\nu:[a,b]\to Y$ on $Y$ we have
$\Re\big(\int_\nu \beta\big)=0$.
\item[(C2)] $\beta$ is nowhere vanishing, i.e. $\sum_{j=1}^k\beta_j\bar\beta_j$ is a volume form on $Y$.

\item [(C3)] $\sum_{j=1}^n \beta_j^2=0$.
\end{enumerate}

 Fix $y_0\in Y$. Then the formula
\begin{equation}\label{phi-in-terms-of-omega}
\varphi(y)=\Re \big( \int_{\nu_y} \beta	\big)
\end{equation}
(where $\nu_y$ is any smooth path joining $y_0$ to $y$ in $Y$) defines a minimal conformal  immersion $\varphi:Y\to \R^n$.

Conversely, any minimal conformal immersion $\varphi:Y\to \R^n$  is given by   (\ref{phi-in-terms-of-omega}) for a  $\C^n$-valued holomorphic 1-form $\beta$ on $Y$ satisfying   conditions (C1)-(C3).
\end{thry}

The second part of this statement is a holomorphic representation theorem for minimal conformal immersion in the Euclidian space. For $n=3$ this result can be further refined giving an explicit  construction method -- called the Weierstrass representation theorem   --  for minimal conformal immersions on simply connected Riemann surfaces (see for instance \cite[Folgerung 3.36]{Ku}, \cite[Satz 8.5.1]{EsJo}).
\vspace{1mm}

Let $\varphi:Y\to \R^n$ be a smooth map. The existence of a $\C^n$-valued holomorphic 1-form $\beta$ on $Y$ satisfying (C1) for which (\ref{phi-in-terms-of-omega}) holds is equivalent to the existence of a commutative diagram
\begin{equation}\label{diag-intro}
 \begin{tikzcd}
 \tilde Y \ar[r, "\tilde f"] \ar[d, "c"'] & \C^n\ar[d, "\Re"]\\
 Y\ar[r, "\varphi"]& \R^n	
 \end{tikzcd},  
\end{equation}
where $c$ is a covering map of Riemann surfaces and $\tilde f$ is a  holomorphic map.  If $\varphi:Y\to \R^n$  is a minimal conformal immersion, then  for any such diagram the form $\beta=d\tilde f$ also satisfies conditions (C2)-(C3), in other words $\tilde f$ is a holomorphic immersion which is isotropic (null) with respect to the standard symmetric form $\omega_0\coloneq \sum_{j=1}^n dz_j^2$   on $\C^n$ (see section \ref{isotropic-section} in this article).  

Theorem \ref{Th1-intro} can be formulated as follows:

\begin{thry}\label{Th2-intro} 
Let $f:Y\to \C^n$ be an $\omega_0$-isotropic holomorphic immersion. Then $\Re(f)$ is a minimal conformal immersion. 

Conversely, let $\varphi:Y\to \R^n$  is a minimal conformal immersion. Then $\varphi$ fits in a commutative diagram of the form (\ref{diag-intro}) with $\tilde f$ holomorphic. For any such commutative diagram,   $\tilde f$ is an $\omega_0$-isotropic holomorphic immersion.
 \end{thry}

Bryant's renowned theorem gives a similar result for conformal CMC-1 immersions in the hyperbolic 3-space $\H^3$  \cite[Theorem A]{Br}. Let 
$$\pi:\SL(2,\C)\to \H^3=\SL(2,\C)/\SU(2)$$
 be the canonical projection and endow   $\H^3$ with the standard hyperbolic metric. Let $\omega$ be the holomorphic bi-invariant symmetric bilinear form on $\SL(2,\C)$ which coincides with the Killing form on the Lie algebra $T_{I_2}\SL(2,\C)=\sl(2,\C)$.  Then 

\begin{thry} \label{Th3-intro} 
Let $f:Y\to \SL(2,\C)$ be an $\omega$-isotropic holomorphic immersion. Then $\pi\circ f$ is a CMC-1 conformal immersion. Conversely, any CMC-1 conformal immersion $\varphi:Y\to\R^n$ fits in a commutative diagram
\begin{equation}\label{diag-intro-Bryant}
 \begin{tikzcd}
 \tilde Y \ar[r, "\tilde f"] \ar[d, "c"'] & \SL(2,\C)\ar[d, "\pi"]\\
 Y\ar[r, "\varphi"]& \H^3	
 \end{tikzcd},  
\end{equation}
where $c$ is a covering map of Riemann surfaces and $\tilde f$ is holomorphic. For any such commutative diagram, $\tilde f$ is an $\omega$-isotropic holomorphic immersion.
\end{thry}

In  \cite{AiAk}, \cite{Fu}, \cite{Le} the authors prove and study a relativistic version of Bryant representation theorem in which the compact real form $\SU(2)$ of $\SL(2,\C)$ is replaced by the split real form $\SL(2,\R)$ and the hyperbolic space $\H^3$  by the de Sitter space $dS_3$.

The goal of this article is to give a ``universal" generalisation of these results in the framework of principal bundles. Let $H$ be a (non necessarily connected) real Lie group,  $\pi:P\to M$ a principal $H$-bundle with $\dim(H)=\dim(M)$, $V\subset T_P$ its vertical subbundle and $\alpha\in A^1(P,\hg)$ a tensorial 1-form of type $\Ad$  which is admissible, i.e. for any $\xi\in P$, the induced  map $\bar\alpha_\xi:T_\xi P/V_\xi\to \hg$    is an isomorphism (see section \ref{J-alpha-A-section}). Let also $A$ be a connection on $P$. The pair $(\alpha,A)$ defines an almost complex structure $J^\alpha_A$ on $P$ which intertwines the subundles $A$ and $V$. By \cite[Theorem 1.1]{Ze}, $J^\alpha_A$  is integrable if and only $(\alpha,A)$ is a solution of a  gauge-invariant non-linear first order differential system.

Let $g:\hg\times\hg\to\R$ be a non-degenerate $\Ad$-invariant symmetric bilinear form and  $g^\alpha_M$, $\g^\alpha_A$ the associated pseudo-Riemannian metrics on $M$, $P$ respectively. We also obtain a bilinear form $\omega^{\alpha,g}_A:T_P\times_P T_P\to \C$ which is of type (1,0) with respect to $J^\alpha_A$ and is holomorphic if $J^\alpha_A$ is integrable. The tensorial form $\alpha$ defines a bundle isomorphism $T_M\textmap{\simeq} \Ad(P)$, so it induces a Lie bracket $[\cdot,\cdot]_\alpha:T_M\times_M T_M\to T_M$ which makes  $T_M$ a Lie algebra bundle.

Assume that the pair $(\alpha,A)$ satisfies the gauge invariant equations which characterise the integrability of  $J^\alpha_A$  (see \cite[Theorem 1.1]{Ze} and Theorem \ref{mainZe} in this article). Our main result is Theorem \ref{main} which establishes a correspondence -- similar to Theorems \ref{Th2-intro}, \ref{Th3-intro} -- between space-like $\omega^{\alpha,g}_A$-isotropic holomorphic immersions $f:Y\to P$ and space-like conformal immersions $\varphi:Y\to M$ whose mean curvature vector field is given by the simple formula 
\begin{equation}\label{Hg-general-formula-intro} 
\vec{\Hg}_\varphi= \varphi^*([\cdot,\cdot]_\alpha)/\vol_\varphi.	
\end{equation}
In this formula the Lie bracket $[\cdot,\cdot]_\alpha$ is regarded as an element of $A^2(M,T_M)$, $\varphi^*([\cdot,\cdot]_\alpha)\in A^2(Y,\varphi^*(T_M))$ stands for its pull-back via $\varphi$ and $\vol_\varphi$ stands for the volume form associated with the canonical orientation of $Y$ and the metric $\varphi^*(g^\alpha_M)$. Therefore we have a Bryant type correspondence for any 4-tuple $(P\textmap{\pi}M,\alpha,A,g)$, where $\alpha$ is an admissible tensorial 1-form of type $\Ad$ on $P$ and $A$ is a connection on $P$ such that $(\alpha,A)$ satisfies the above mentioned integrability equations, and $g:\hg\times\hg\to\R$ is a non-degenerate $\Ad$-invariant symmetric bilinear form.

A large class of examples of such 4-tuples is obtained as follows: let $G$ be a connected complex Lie group, and let $H\subset G$ be a real form of $G$, i.e. a (non-necessarily connected) closed subgroup whose Lie algebra $\hg$ is a real form of $\g$. The  canonical left-invariant $\g$-valued 1-form  $\eta\in A^1(G,\g)$  on $G$ decomposes as $\eta=\theta-i\alpha$, where $\theta=\Re(\eta)$, $\alpha=-\Im(\eta)$ with respect to the real structure $\hg$ of $\g$. It is easy to see that $\alpha$ is an admissible tensorial 1-form of type $\Ad$ on $G$ , $\theta$ is the connection form of a left-invariant connection $A$ on $G$ (regarded as principal bundle over $M\coloneq G/H$), and $J^\alpha_A$ coincides with the canonical almost complex structure of $G$, so is integrable. The pseudo-Riemannian  metric $g^\alpha_M$ coincides with the canonical pseudo-Riemannian  metric $g_M$ induced by $g$ via $\alpha$, and   the symmetric form $\omega^{\alpha,g}_A$ associated with   a non-degenerate $\Ad_H$-invariant symmetric bilinear form  $g:\hg\times\hg\to\R$ is just   the unique left-invariant form $\omega^g_H$ which coincides at the unit element $e\in G$ with the $\C$-bilinear extension  $g^\C:\g\times\g\to\C$ of $g$ (see section \ref{RealForms-subsection} for details). The obtained Lie bracket $[\cdot,\cdot]_\alpha$ on the tangent bundle $T_M$ is determined by $H$ and will be denoted by $[\cdot,\cdot]_H$.

Therefore we obtain a Bryant type correspondence, stated in Theorem \ref{main-real-forms}, for any triple $(G,H,g)$ consisting of a complex Lie group $G$, a real form $H$ of $G$ and a non-degenerate $\Ad_H$-invariant symmetric bilinear form $g$ on $\hg$.    

More generally, let $\Gamma\subset G$ be a discrete subgroup acting freely and properly discontinuously from the left on $M=G/H$. We obtain a principal $H$-bundle 
$$\pi_\Gamma:\Gamma\backslash G\to \Gamma\backslash G/H=\Gamma\backslash M.$$
The objects  $\alpha$, $A$, $\omega^g_H$ descend from $G$ to $\Gamma\backslash G$. Similarly,   the pseudo-Riemannian metric $g_M$ and the Lie bracket $[\cdot,\cdot]_H:T_M\times_M T_M\to T_M$ descend to $\Gamma\backslash M$.   The general Theorem \ref{main} gives a Bryant type correspondence for any 4-tuple $(G,H,g,\Gamma)$ as above (see section \ref{Bryant-for-Gamma-section}).\\

The article is structured as follows: in section \ref{prelim-section} we present briefly the basic notions, the formalism and the notations used in the article. In section \ref{J-alpha-A-section} we explain our principal bundle formalism, in particular we introduce the almost complex structure $J^\alpha_A$ mentioned above. Section \ref{main-th-section} is dedicated to the main theorem, to its corollaries (Corollaries \ref{main-coro}, \ref{scalar-mean-curv-coro}) and to its proof which is explained in detail in section \ref{ProofSection}. In section \ref{Weierstrass-repr-section}, inspired by the classical Theorems \ref{Th2-intro}, \ref{Th3-intro} and by the terminology used in \cite[Section 6.1]{BDHH}, we introduce the concept (simple) {\it Weierstrass representation} of a space-like conformal immersion $\varphi:Y\to M$ whose mean curvature vector field is given by formula (\ref{Hg-general-formula-intro}). A Weierstrass representation of $\varphi$ is a commutative diagram
$$
 \begin{tikzcd}
 \tilde Y \ar[r, "\tilde f"] \ar[d, "c"'] & P\ar[d, "\pi"]\\
 Y\ar[r, "\varphi"]& M	
 \end{tikzcd}, \eqno{(W)}
 $$
 where $c$ is a covering map of Riemann surfaces and $\tilde f$ is holomorphic. A simple Weierstrass representation of $\varphi$ is just a holomorphic lift $  f:Y\to P$ of $\varphi$. For any (simple)  Weierstrass representation of $\varphi$, the map $\tilde f$ (respectively $f$) is a space-like $\omega^{\alpha,g}_A$-isotropic holomorphic immersion.   The second part of our main theorem can be reformulated as follows: Any  space-like conformal immersion $\varphi:Y\to M$ from a (simply) connected Riemann surface and whose mean curvature vector field is given by formula (\ref{Hg-general-formula-intro}) admits a (simple) Weierstrass representation.

 Section \ref{RealForms-section} is dedicated to the generalised  Bryant  correspondences for principal bundles of the form $\pi_H:G\to G/H\eqcolon M$, where $H$ is real form of $G$. The main result (Theorem \ref{main-real-forms}) is explained in section \ref{RealForms-subsection} whereas section \ref{Examples-subsection} is dedicated to examples of this type. We mention here briefly Example \ref{Example-reductive} which concerns  the case when $G$ is a reductive complex Lie group. In this case one  considers     a Weyl involution $\theta$ of $G$, and   a Cartan involution $\tau$  which commutes with $\theta$. Consider the compact real form   $K\coloneq G^\tau$  and the split real form   $K'\coloneq G^{\theta\tau}$ of $G$. Let  $g$  be inner product on $\kg$ which is $\Ad_K$-invariant  and $\theta$-invariant,  let $g^\C:\g\times\g\to \C$  be its $\C$-bilinear extension and $g'$ the restriction $g'$ of $g^\C$ to $\kg'\times\kg'$, which is an $\R$-valued, non-degenerate symmetric bilinear form. We obtain two triples $(G,K,g)$, $(G,K',g')$ as above; the  holomorphy and the isotropy conditions intervening in the  associated Bryant type correspondences coincide. The first  quotient $G/K$ is  a Riemannian manifold, whereas $G/K'$ is a pseudo-Riemannian with signature $-\rk(G)$.

 We will look in more detail at the case $G=\GL(n,\C)$, $K=\U(n)$, $K'=\GL(n,\R)$, $g(a,b)=-\Tr(ab)$. The quotients  $G/K$,  $G/K'$ have interesting geometric interpretations: the former can be identified with the manifold of Hermitian inner products,  and the latter  with the manifold of real structures on  $\C^n$. We compute explicitly the Lie  brackets  $[\cdot,\cdot]_K$, $[\cdot,\cdot]_{K'}$ on the tangent bundles of these manifolds, so we have  an explicit formula for the mean curvature vector field of the conformal immersions which intervene in each of the two Bryant type correspondences. We conclude this section with an infinite dimensional example in which the role of the homogeneous manifold $M$ is played by the space ${\cal M}^E$ of Hermitian metrics on a complex vector bundle over a compact oriented Riemannian manifold. 
 
 In section \ref{Bryant-for-Gamma-section} we discuss the generalised  Bryant correspondence for bundles of the form  $\pi_\Gamma: P_\Gamma\coloneq \Gamma\backslash G\to \Gamma\backslash M\eqcolon M_\Gamma$ , where $\Gamma\subset G$ a discrete group of $G$ acting freely and properly discontinuously on $G/H$.  We consider in more detail the classical case when $G=\SL(2,\C)$, $H=\SU(2)$ and $\Gamma$ is a lift to $\SL(2,\C)$ of a torsion free Kleinian subgroup $\Gamma_0\subset  \PSL(2,\C)$.
 
 As noticed in \cite[Section 6.1]{BDHH} the existence of a simple Weierstrass representation of a CMC-1 conformal immersion $\varphi:Y\to M_\Gamma$  is especially interesting when both the Riemann surface $Y$ and the quotient $\Gamma\backslash\SL(2,\C)$ (or, equivalently, $\Gamma\backslash \H^3$) are compact.
 If a simple Weierstrass representation of a CMC-1 conformal embedding $\varphi:Y\hookrightarrow\Gamma\backslash\H^3$ exists, then one obtains an embedding $f:Y\hookrightarrow \Gamma\backslash \SL(2,\C)$ which is not only holomorphic, but also  isotropic with respect to the  symmetric  form  on $P_\Gamma$ induced  by the Killing form of $\sl(2,\C)$. The existence of such an embedding would be a substantial strengthening of the main result of \cite {BDHH} which states that certain compact quotients of the form  $\Gamma\backslash \SL(2,\C)$ do admit higher genus holomorphic curves, so gives a positive  answer of an old problem  formulated  by Winkelmann \cite{Wi}.

\section{Preliminaries and notations} \label{prelim-section}

\subsection{Connections and exterior covariant differrentiation}\label{connections-section}

Let $M$ be a differentiable manifold, $H$ be a Lie group,  $P\textmap{\pi} M$ be a principal $H$-bundle,  $T_P\textmap{p} P$ its tangent bundle, 
$$T_P\supset V\textmap{p_V} P$$
 its vertical subbundle. For any $\xi\in P$ we have a standard isomorphism $\#_\xi:\hg\textmap{\simeq} V_\xi$ induced by the infinitesimal action of $\hg$ on $P$. The family $(\#_\xi^{-1})_{\xi\in P}$ defines a smooth map $\vartheta:V\to\hg$ whose restrictions to the fibres are isomorphisms.

By a connection on $P$  we mean a  $H$-invariant  distribution $A\subset T_P$ which defines a complement subbundle of $V$ in the tangent bundle $T_P$ (see \cite[section II.1]{KN}). We will denote by ${\cal A}(P)$ the set of connections on $P$. 
The connection form of a connection $A\in {\cal A}(P)$ is the unique $\hg$-valued 1-form $\theta$ on $P$ which vanishes on $A$ and coincides with $\vartheta$ on $V$.  The connection form is a pseudo-tensorial form of type $\Ad$ (see \cite[section II.5]{KN}).  The curvature $\Omega_A$ is given by the formula $\Omega_A=d\theta+\frac{1}{2}[\theta\wedge \theta]$; it is a tensorial 2-form of type $\Ad$.

Let $\rho:H\to \GL(F)$ be a linear representation of $H$ on a finite dimensional real vector space $F$, and let $E^\rho_P\coloneq P\times_\rho F$ be the associated vector bundle.  

We will denote by $A^k_\rho(P,F)$ the space of tensorial $F$-valued forms of type $\rho$ on $P$ (see again \cite[section II.5]{KN})). Recall that for any $k\in \N$ we have a canonical isomorphism
\begin{equation}\label{StandardIdent}
A^k_\rho(P,F) \textmap{\simeq}A^k(M,E^\rho_P),	
\end{equation}
which identifies the space of tensorial $F$-valued forms on $P$ with the space of $E^\rho_P$-valued $k$-forms on $M$ (see \cite[Proposition 3.2.9]{Te}, \cite[Example II.5.2]{KN}). For a tensorial form $\alpha\in A^k_\rho(P,F)$, we will denote by $\tilde\alpha$  the corresponding element of $A^k(M,E^\rho_P)$. In particular, for a connection $A\in {\cal A}(P)$, we put
$$
F_A\coloneq \tilde\Omega_A\in A^2(\Ad(P)),
$$
where $\Ad(P)\coloneq E^\Ad_P$ is the vector bundle associated with $P$ and the adjoint representation $\Ad:H\to \GL(\hg)$ of $H$ on its Lie algebra.

Let $A\in {\cal A}(P)$ and  let 
$$D_A^\rho:A^k_\rho(P,F)\to A^{k+1}_\rho(P,F)$$
 be associated the exterior covariant differrentiation \cite[section II.5]{KN}. The set ${\cal A}(P)$ of connections on $P$ has a natural structure of an affine space with model vector space $A^1_\Ad(P,\hg)$ and translations in this affine space change the curvature as follows:   
\begin{equation}\label{Omega(A+alpha)}
\Omega_{A+\alpha}=\Omega_A+D_A\alpha+\frac{1}{2}[\alpha\wedge\alpha]	.
\end{equation}

 Recall that a connection $A\in {\cal A}(P)$ defines a linear connection $\nabla^\rho_A$ on the associated vector bundle $E^\rho_P$, and, via the  identifications (\ref{StandardIdent}), $D_A^\rho$ corresponds to the de Rham operator $d_A^\rho:A^k(M,E^\rho_P)\to A^{k+1}(M,E^\rho_P)$ associated with $\nabla^\rho_A$ (see \cite[Remarque 3.2.11]{Te}). In the special case   $\rho=\Ad$, we will write $D_A$, $d_A$ instead of  $D_A^\Ad$, $d_A^\Ad$ to save on notations.
 
 A point $\xi\in P$ defines 
 \begin{itemize}
 \item An $H$-equivariant diffeomorphism $l_\xi:H\textmap{\simeq} P_{\pi(x)}$ given by $h\mapsto \xi h$.
 \item 	A linear isomorphism $l_\xi^\rho: F\textmap{\simeq}E^P_{\rho,\pi(\xi)}$ given by $u\mapsto [(\xi,u)]$.
 \end{itemize}
 We will use the notations
 \begin{equation}\label{isomorph-with-xi}
 l^\xi\coloneq (l_\xi)^{-1}: P_{\pi(x)}\to H,\ 	l^\xi_\rho\coloneq (l_\xi^\rho)^{-1}:E^P_{\rho,\pi(\xi)}\to F
 \end{equation}
for the inverse isomorphisms. Similarly, a local section $\tau\in \Gamma(U,P)$ defines a smooth map
$$l_\rho^\tau: E^P_{\rho|U}\textmap{\simeq} F$$
(whose restrictions to the fibres are isomorphisms) which yields vector space isomorphisms (denoted by the same symbol to save on notations)
\begin{equation}\label{isom-l-tau-rho}
l^\tau_\rho:  A^k(U,E^\rho_P)\textmap{\simeq} A^k(U,F).
\end{equation}
Note the  identity
\begin{equation}\label{obv-id}
l^\tau_\rho(\tilde\alpha)=\tau^*(\alpha) \ \forall \alpha\in A^k_\rho(P,F).	
\end{equation}
Via the isomorphisms (\ref{isom-l-tau-rho}), the de Rham operator  $d_A^\rho$ can be written as $d+[\tau^*(\theta) \wedge\cdot]_\rho$, where $[\cdot,\cdot]_\rho:\hg\times F\to F$ is the bilinear map associated with the Lie algebra morphism $\rho_e:\hg\to \End(F)$. More precisely we have

\begin{equation}\label{dA-with-tau}
l^\tau_\rho\circ d_A^\rho=(d+[\tau^*(\theta)\wedge\cdot]_\rho)\circ l^\tau_\rho.	
\end{equation}

We will need the following definition: Let $E\to M$ be real vector bundle on $M$,  and  $J$ be an almost complex structure on $M$, and let $\lambda\in A^1(M,E)$ be an $E$-valued 1-form. We define the 1-form $J\lambda \in A^1(M,E)$ by
\begin{equation}\label{Def-Jlambda}
J\lambda\coloneq \lambda\circ J.	
\end{equation}

 The correspondence $\lambda\mapsto J\lambda$ obviously commutes with composition from the left by vector bundle isomorphisms.

\subsection{Isotropic maps}
\label{isotropic-section}

Let $X$ be a differentiable manifold, $W$ a real vector space, and $\omega:T_X\times_X T_X\to W$ a smooth map whose restriction the fibres is symmetric and  bilinear.

\begin{dt} \label{isotropy-def}
A differentiable map $f:Y\to X$ will be called $\omega$-isotropic ($\omega$-null) if for any $y\in Y$ and any $v\in T_yY$ we have $\omega(f_*(v), f_*(v))=0$. 	
\end{dt}

Now let $J$ be an almost complex structure on $X$, and let $\omega:T_X\times_X T_X\to \C$ be a smooth, symmetric map which is $\C$-bilinear with respect to $J$. Therefore, for any $x\in X$ we have a symmetric  form $\omega_x:T_xX\times T_xX\to \C$ which is $\C$-bilinear with respect to $J_x$. It follows that the isotropic (null) cone $C_x^\omega\subset T_xX$ of $\omega_x$ is $J_x$-invariant. In this case a differentiable map $f:Y\to X$ is $\omega$-isotropic if and only if for any $y\in Y$ and any $v\in T_yY$ we have $f_*(v)\in C^\omega_x$.

We will need this condition in the special case when $J$ is integrable, $\omega$ is holomorphic, $Y$ is a Riemann surface and $f:Y\to X$ is a holomorphic map. 

\begin{re}\label{isotropy-cond}
Let $J_X$ be an integrable almost complex structure on $X$, and let
$$\omega:T_X\times_X T_X\to \C$$
 be a   map  whose restriction to the fibres is symmetric and $\C$-bilinear and which is holomorphic with respect to $J_X$. Let $(Y,J)$ be  a Riemann surface and  $f:Y\to X$ a holomorphic map with respect to $(J,J_X)$. 	Then $f^*(\omega)$ can be regarded as a holomorphic section of the holomorphic line bundle ${\cal K}_Y^{\otimes 2}$ and $f$ is $\omega$-isotropic if and only if this section vanishes.
\end{re}

In particular, when $Y$ is closed, then the obstruction to the $\omega$-isotropy of a holomorphic map $f:Y\to X$ belongs to the finite dimensional complex vector space $H^0(Y,{\cal K}_Y^{\otimes 2})$.

\subsection{Non-degenerate and space-like immersions in pseudo-Riemannian manifolds. The mean curvature vector field}\label{immersions-section}
 {\ }\vspace{-2mm}\\

Let $(X,g)$ be pseudo-Riemannian manifold. An immersion $f:Y\to X$ will be called non-degenerate (with respect to $g$) if $f^*(g)$ is non-degenerate at any point. Equivalently, $f$ is non-degenerate if, for any $y\in Y$, we have $f_{*y}(T_yY)\cap N^f_y=\{0\}$, where 
$$
N^f_y\coloneq f_{*y}(T_yY)^{\bot_g}\subset T_{f(y)}X
$$
 is the normal space of $f$ at $y$ with respect to $g_{f(y)}$. If $f$ is non-degenerate, then $f^*(g)$ is a pseudo-Riermannian metric on $Y$ and $f:(Y,f^*(g))\to (X,g)$ becomes an isometric immersion.

Suppose that $f$ is non-degenrate. For any $y\in Y$, $N^f_y$ is a complement of $f_{*y}(T_yY)$ in $T_{f(y)}X$ and we have a direct sum  decomposition 
\begin{equation}\label{DirectSum}
f^*(T_X)=T^f\oplus N^f,
\end{equation}
where   $T^f\subset f^*(T_X)$ is the image of the bundle monomorphism $T_Y\to f^*(T_X)$ defined by $f_*$, and  
$$N^f\coloneq\coprod_{y\in Y}N^f_y\subset f^*(T_X)$$
is the normal subbundle of $f$.   
%


Let $\nabla^g$ be the Levi-Civita connection of $g$ (see for instance \cite[section 1.2]{Ch1} for the pseudo-Riemannian case) and let $\nabla\coloneq f^*(\nabla^g)$ be its pull back to the bundle $f^*(T_X)$ over $Y$ (see \cite[p. 28]{Te}). Using the direct sum decomposition (\ref{DirectSum}) we can write 
$$
\nabla =\begin{pmatrix}
\nabla^f & -a^*\\
a &	\nabla^N
\end{pmatrix},
$$
where: 
\begin{itemize}
	\item $\nabla^f$ is the induced connection on $T^f$; this connection corresponds to the Levi-Civita connection of the pseudo-Riemannian manifold $(Y,f^*(g))$ via the obvious bundle isomorphism $T_Y\textmap{\simeq} T^f$,
	\item $\nabla^N$ is the induced connection on the normal bundle $N^f$,
	\item  $a\in A^1(Y,\Hom(T^f,N^f))$ is the second fundamental form of $\nabla$ with respect to the decomposition (\ref{DirectSum}), 
	\item for any $y\in Y$ and $v\in T_Y$,  $a^*(v)\in \Hom (N^f_y,T^f_y)$ is the adjoint of $a(v)$ with respect to the symmetric  non-degenerate  bilinear forms induced by $g$ on $T^f_y$ and $N^f_y$.\end{itemize}

It is well known (see for instance \cite[Proposition 2.1 p. 27]{Ch1}) that for any $y\in Y$ the bilinear form $h_y:T_yY\times T_yY\to N^f_y$ defined by 
$$
h_y(u,v)\coloneq a(u)(f_*(v))
$$
is symmetric, so can be viewed as an element of $S^2(T_y^*Y)\otimes N^f_y$. The space  $S^2(T_y^*Y)$  can be identified in the usual way with the space $\End_{\rm s}(T_yY)$ of $f^*(g)_y$-symmetric (self-adjoint) endomorphisms of $T_yY$, so $h_y$ defines  an element $l_y\in\End_{\rm s}(T_yY)\otimes N^f_y$. 

We refer to  \cite[section 2.1]{Ch2} for the following definition which will play an important role in this article: 
\begin{dt}\label{Hg-def}The mean curvature vector of $f$ at $y$  is defined by 
$$
\vec\Hg_f(y)\coloneq \frac{1}{\dim(Y)} \Tr_{_{T_yY}}(l_y)\in N^f_y.
$$
The mean curvature vector field of $f$ is the section $\vec\Hg_f$ of $N^f$ defined by $y\mapsto \vec\Hg_f(y)$.
\end{dt}

Let $m\coloneq \dim(Y)$ and let $V_1$,\dots,$V_m$ be a vector fields on an open set  of $U\subset Y$ which give an $f^*(g)_y$-orthonormal basis at any point $y\in U$ in the sense that
\begin{equation}\label{g-orthon}
\langle V_i,V_j\rangle_{f^*(g)}=\delta_{ij}\varepsilon_i, \hbox{ where }\varepsilon_i\in\{\pm 1\}.
\end{equation}
Let $V_i^f$ be the section of $T^f|_U$  defined by the map $y\mapsto  f_{*y}(V_i(y))$.
Then
\begin{equation}\label{expl-formula-for-Hg}
\vec\Hg_f|_U =\frac{1}{m} \sum_{i=1}^m \varepsilon_i h_y(V_i,V_i)=\frac{1}{m} \sum_{i=1}^m \varepsilon_i a(V_i)(V_i^f)=\frac{1}{m} \sum_{i=1}^m \varepsilon_i(\nabla_{V_i}V_i^f)^{\bot_g},
\end{equation}
where $\bot_g$ stands for the projection on the second summand in (\ref{DirectSum}). 

\begin{dt}
Let $(X,g)$ be pseudo-Riemannian manifold. An immersion $f:Y\to X$ will be called space-like   if	$f^*(g)$ is positive definite.
\end{dt}

\begin{re} \label{Hg-for-space-like} Any space-like immersion is obviously non-degenerate. If $f$ is space-like,  we have $\varepsilon_i=1$   in the formulae (\ref{g-orthon}), (\ref{expl-formula-for-Hg}).
\end{re}

We will need the following natural

\begin{dt}
Let $(X,g)$ be pseudo-Riemannian manifold, $Y$  an oriented  manifold, and $f:Y\to X$ a space-like immersion. 	We will denote by $\vol_f$  the volume form of the oriented Riemannian manifold $(Y,f^*(g))$.
\end{dt}

\section{The almost complex structure  associated with an admissible tensorial 1-form  and a connection}
\label{J-alpha-A-section} 

\subsection{The almost complex structure \texorpdfstring{$J^\alpha_A$}{J}. The integrability condition}\label{J-alpha-A-subsection}

Let  $P\textmap{\pi} M$ be a principal $H$-bundle,  $T_P\textmap{p} P$ its tangent bundle, $T_P\supset V\textmap{p_V} P$ its vertical subbundle. The canonical  map $\vartheta:V\to\hg$ (see section \ref{connections-section}) defines a bundle isomorphism 
$$
(p_V,\vartheta):V\textmap{\simeq} P\times \hg,
$$
which can be regarded as a trivialisation of $V$ with standard fibre $\hg$.

\begin{dt}
A tensorial 1-form $\alpha\in A^1_\Ad(P,\hg)$	of type $\Ad$ on $P$   will be called admissible if, for any $\xi\in P$, the  map $\bar\alpha_\xi:T_\xi P/V_\xi\to \hg$ induced by $\alpha_\xi$ is an isomorphism.
\end{dt}

Note that  admissible forms can exist only if $\dim(M)=\dim(\hg)$. An admissible  tensorial 1-form $\alpha$ defines a smooth $\bar\alpha:T_P/V\to \hg$ whose restriction to the fibres are isomorphism. Denoting by $\bar p: T_P/V\to P$ the bundle projection, we obtain a bundle isomorphism 
$$(\bar p,\bar\alpha):T_P/V\textmap{\simeq} P\times\hg,$$
which can be regarded as a trivialisation of  the quotient bundle $T_P/V$ with standard fibre $\hg$. The composition $j^\alpha:(p_V,\vartheta)^{-1}\circ (\bar p,\bar\alpha):T_P/V\textmap{\simeq} V$ is  a $H$-invariant bundle isomorphism  which descends via $H$-factorisation to a bundle isomorphism $j^\alpha_M: T_M\to \Ad(P)$ on $M$, which is given explicitly by $T_xM\ni v\mapsto \tilde \alpha_x(v)$, where $\tilde\alpha\in A^1(M,\Ad(P)$ is the $\Ad(P)$-valued 1-form associated with $\alpha$ (see section \ref{connections-section}).	Via this isomorphism, $T_M$ becomes a bundle of Lie algebras; we will denote by 
$$[\cdot,\cdot]_\alpha:T_M\times_M T_M\to T_M$$
the fibrewise bilinear map  induced via  $j^\alpha_M$  by the canonical fibrewise bilinear map $[\cdot,\cdot]:\Ad(P)\times_M \Ad(P)\to \Ad(P)$ defined by the Lie algebra bracket on $\hg$.

Now let  $A\subset T_P$  be a connection on $P$. Identifying $A$ with $T_P/V$ in the obvious way, we obtain 
 \begin{itemize}
 \item a smooth map $\alpha_A:A\to \hg$ induced by $\bar\alpha$ whose restriction to the fibres are isomorphisms,
 \item  a bundle isomorphism $(p_A,\alpha_A):A\to P\times \hg$,	i.e. a trivialisation of $A$ with standard fibre $\hg$,
 \item a bundle isomorphism 
$j^\alpha_A=(p_V,\vartheta)^{-1}\circ (p_A,\alpha_A):A\textmap{\simeq} V$
 induced by $j^\alpha$ identifying the horizontal and the vertical subbundles.
 \end{itemize}

Recall that the canonical map $\vartheta:V\to\hg$  (whose restrictions to the fibres are isomorphisms) coincides with the restriction $\theta|_V$ of the connection form $\theta$ of $A$ \cite[section II.1]{KN}.

Denoting by  $q:T_M\to M$, $\qg: \Ad(P)\to M$  the  obvious bundle projections, we see that the data of a pair $(\alpha,A)$ consisting of an admissible 1-form and a connection $A\subset T_P$ gives the following commutative diagram in which the blue (red) arrows denote bundle isomorphisms over $P$ (respectively $M$), and the purple arrows denote fibrewise isomorphic, $\pi$-covering bundle morphisms:
\vspace{-6mm}
\begin{equation}\label{diag}\hspace*{-18mm}	 
\begin{tikzcd}[row sep=8ex, column sep=1ex]
&\color{blue} P\times\hg  \ar[lddddd, magenta, bend right=90] \ar[rddddd,  magenta, bend left=90] \ar[ddd,  "p_1" description,   bend right=10] &\\
\color{blue}  A\ar[d, blue, hook'] \ar[rr, blue,  crossing over,  "\ \ \ j^\alpha_A\simeq"'] \ar[ru, blue, bend left=10,  "\simeq(p_A{,}\alpha_A)" description  ]  \ar[dddd, magenta,  bend right=25]   & &  \color{blue} V\ar[d, blue, hook'] \ar[lu, bend right=10, blue,  "(p_V{,}\vartheta)\simeq" description]  \ar[dddd, magenta,  bend left=25] 	\\
\color{blue} T_P\ar[rr, blue, crossing over, "\ \ \ J^\alpha_A\simeq", ] \ar[ruu, blue, two heads, crossing over,  "{(p,\alpha)}" description, near start] \ar[rd, "p"'] && \color{blue}T_P \ar[ld, "p"] \ar[luu, blue, two heads, crossing over,   "{(p,\theta)}" description, near start] \\
 & P \ar[d, brown, two heads, "\pi"] & \\
  & M=P/H & \\
\color{red} A/H=T_M \ar[ru, "q"] \ar[rr, red, "j^\alpha_M\simeq"]   & &{\color{red} \Ad(P)=V/K}\,. \ar[lu, "\qg"']   
\end{tikzcd}
\end{equation}

We associate  with such a pair $(\alpha,A)$ the almost complex structure $J^\alpha_A$ on $P$ given by the formula
$$J^\alpha_A(u+v)= -(j^\alpha_{A,\xi})^{-1}(v)+j^\alpha_{A,\xi}(u)
$$
for any $(u,v)\in A_\xi\times V_\xi$. 

We denote by $\nabla^\Ad_A$ the linear connexion defined by $A$ on the associated bundle $\Ad(P)=P\times_\Ad\hg$ (see section \ref{connections-section}), and by $\nabla^\alpha_A$ the linear connection on $T_M$ induced by $\nabla^\Ad_A$ via the isomorphism $j^\alpha_M$.

  Recall that the space ${\cal A}(P)$ of connections on $P$ is an affine space with model vector space the space $A^1_\Ad(P,\hg)$ of tensorial 1-forms of type $\Ad$.

We refer to \cite{Ze} for the following   result which will play an important role in our proofs:
\begin{thry} \label{mainZe} \cite[Theorem 1.1, Proposition 2.1, Proposition 2.2]{Ze}
Let $\alpha$ be an admissible tensorial 1-form of type $\Ad$ and $A$ a connection on $P$.  Then  
\begin{enumerate}	
\item $J^\alpha_A$ is integrable if and only if the pair $(\alpha,A)\in A^1_\Ad(P,\hg)\times {\cal A}(P) $ is a solution of the differential system	
\begin{equation}\label{ZSys}
\left\{\begin{array}{ccc}
        D_A \alpha & = & 0 \\
        \Omega_A & = & \frac{1}{2} \ [\alpha \wedge \alpha],
\end{array}\right.	 
\end{equation}
where $\Omega_A\in A^2_\Ad(P,\hg)$ denotes the curvature form of $A$.
\item Suppose that $J^\alpha_A$ is integrable. Then the induced connection   $\nabla^\alpha_A$ on $T_M$ is torsion free and its curvature is given  by formula
$$
R(u,v)w  = [[u,v]_\alpha, w]_\alpha.
$$

 \end{enumerate}
\end{thry}

The following remarks justifies the title of this article:

\begin{re}
The integrability condition  (\ref{ZSys}) is a non-linear quasi-linear system of first order PDE's for the unknown $(\alpha,A)$. 	This system is gauge invariant, i.e. the space of its solutions is invariant with respect to obvious action of the automorphism group (the gauge group) of the bundle $P$ on $A^1_\Ad(P,\hg)\times {\cal A}(P)$. Geometric interpretations of the moduli space of solutions have been described in \cite{AlCTe} in the special cases $H=\SU(2)$, $H=\SO(3)$.
\end{re}

We will need the following

 \begin{lm}\label{eta-holom} et $\alpha\in A^1_\Ad(P,\hg)$ be an admissible tensorial 1-form of type $\Ad$, $A$ a connection on $P$, and $\theta\in A^1(P,\hg)$ its connection form. The $\hg^\C$-valued 1-form on $P$ given by
$$\eta\coloneq \theta-i \alpha\in A^1(P,\hg^\C).$$
has the following properties:
\begin{enumerate}
\item It satisfies the identity
$\eta\circ J^\alpha_A=i\eta$, i.e.  is of type (1,0) with respect to   $J^\alpha_A$.  
\item It defines $\C$-linear isomorphisms 	$(T_\xi P,J^\alpha_{A,\xi})\textmap{\simeq} \hg^\C$ for $\xi\in P$.
\item If $J^\alpha_A$ is integrable, then $\eta$ is a holomorphic $\hg^\C$-valued form of type $(1,0)$ on $P$.
\end{enumerate}
\end{lm}
\begin{proof} 
 The definition of the bundle isomorphism $j^\alpha_A:A\to V$ gives the identity $\theta\circ j^\alpha_A=\alpha_A$, which proves (1). The second claim follows from the first taking into account that $\eta_\xi$ maps isomorphically $A_\xi$ onto $i\hg$ and $V_\xi$ onto $\hg$.
 
 For (3):  Taking into account the definition of $J^\alpha_A$, it follows that any right translation $R_h$ with $h\in H$ is a biholomorphism. Therefore the Lie algebra morphism $\hg\to {\cal X}(P)$ given by $v\mapsto v^\#$ takes values in the subalgebra $\Gamma(P, {\cal T}_P)$ of holomorphic vector fields	 of $P$. Here we denoted by ${\cal T}_P$ the holomorphic bundle on the complex manifold $(P,J^\alpha_A)$ obtained by endowing the real tangent bundle $T_P$ with the complex structure defined by $J^\alpha_A$ and its canonical holomorphic structure.

The universal property of the complexified Lie algebra yields a complex Lie algebra extension $\nu:\hg^\C\to \Gamma(P, {\cal T}_P)$, and (by definition of $J^\alpha_A$) this extension defines a global holomorphic trivialisation $\mu:\hg^\C\times P\textmap{\simeq}T_P$ of the holomorphic tangent bundle ${\cal T}_P$ given by $(u,\xi)\mapsto \nu(u)_\xi$. The   1-forms   $\eta$, $p_1\circ \mu^{-1}\in A^1(P,\hg^\C)$ are both of type $(1,0)$ and coincide on the subbundle $V\subset T_P$, so they coincide.   But $p_1\circ \mu^{-1}$ is holomorphic, because $\mu$ is holomorphic. 
\end{proof}

Triples of the form $(P\textmap{\pi}M,\alpha,A)$ where $\pi$ is a principal $H$-bundle, $\alpha$ is an admissible tensorial 1-form of type $\Ad$ and $A$ a connection on $P$ such that $J^\alpha_A$ is integrable (i.e. such that $(\alpha,A)$ solves the  equations (\ref{ZSys})) play an important role in this article. Results on the classification of such triples can be found in  \cite{AlC} and \cite{AlCTe}.

\subsection{The pseudo-Riemannian metric and the isotropic condition associated with a non-degenerate, \texorpdfstring{$\Ad_H$}{Ad}-invariant symmetric bilinear form} \label{section-for-g-alpha-M-omega}
\label{g-alpha-M-omega-section}

Let $g$ be a non-degenerate, $\Ad_H$-invariant symmetric bilinear form  on the Lie algebra $\hg$ and let $g_V$, $g_\Ad$   be the induced  pseudo-Riemannian structures on the vector bundles $V$, $\Ad(P)$.  We also  obtain a pseudo-Riemannian metric $g^\alpha_M$ on the tangent bundle $T_M$ (so on $M$) induced by $g_\Ad$ via $j^\alpha_M$. 
By the definition of $g^\alpha_M$ we have
\begin{equation}\label{g-compos-pi*}
g^\alpha_M(\pi_*(u),\pi_*(v))=g(\alpha(u),\alpha(v)).
\end{equation}
for any  $\xi\in P$ and $u$, $v\in T_\xi P$.

Let now $A$ be a connection on $P$. Using the isomorphism $j^\alpha_A$ we obtain a pseudo-Riemannian structure $g^\alpha_A$ on $A$. Note that $g_V$ lifts $g_\Ad$ and $g^\alpha_A$ lifts $g^\alpha_M$. We also obtain a product pseudo-Riemannian structure $\g^\alpha_A\coloneq g^\alpha_A\oplus g_V$ on $P$, which is  pseudo-Hermitian with respect to the almost complex structure $J^\alpha_A$ and is given by the explicit formula
\begin{equation}\label{formula-for-g-on-P}
\g^\alpha_A(u,v)=g(\alpha(u),\alpha(v))+g(\theta(u),\theta(v)).
\end{equation}

  Since $g$ is $\Ad$-invariant, it follows that $g_\Ad$ is $\nabla^\Ad_A$-parallel, so $g^\alpha_M$ will be $\nabla^\alpha_A$-parallel. Taking into account Theorem \ref{mainZe} (2) it follows that
  \begin{co}\label{coro-LC}
  Suppose that $J^\alpha_A$ is integrable.  Then $\nabla^\alpha_A$ coincides with the Levi-Civita connection of the pseudo-Riemannian metric $g^\alpha_M$, and, for any	  non-degenerate plane $\sigma=\langle u,v\rangle\subset T_xM$ (see \cite[p. 13]{Ch1}),  the sectional curvature $K(\sigma)$ is given by
$$K(\sigma)=-\frac{g^\alpha_M ([u,v],[u,v]) }{g^\alpha_M ( u,u)g^\alpha_M ( v,v)-g^\alpha_M (u,v)^2}\,.
$$
  \end{co}

Let $g^\C:\hg^\C\times\hg^\C\to\C $ be the $\C$-bilinear extension of $g$ and let
$$ \omega^{\alpha,g}_A:T_P\times_P T_P\to \C$$
be the symmetric $\R$-bilinear form defined by 
$$ \omega^{\alpha,g}_A(u,v)=g^\C(\eta(u),\eta(v)). 
$$
By Lemma \ref{eta-holom} it follows that
\begin{re}
The form $\omega^{\alpha,g}_A$ is $\C$-bilinear with respect to $J^\alpha_A$. In particular, for any $\xi\in P$, the isotropic (null) cone  $C^{\omega^{\alpha,g}_A}_\xi\subset T_\xi P$ is $J^\alpha_A$-invariant.

If $J^\alpha_A$ is integrable, then $\omega^{\alpha,g}_A$, regarded as a symmetric bilinear form on the holomorphic tangent bundle ${\cal T}_P$,  is holomorphic.  The isotropy criterion given by Remark \ref{isotropy-cond} applies.
\end{re}

\subsection{First properties}\label{first-prop-sect} Let $\pi:P\to M$ be a principal $H$-bundle $\alpha\in A^1_\Ad(P,\hg)$ such that for any $\xi\in P$ the map $A_\xi\to \hg$ induced by $\alpha$ is an isomorphism and the integrability conditions  (\ref{ZSys}) are satisfied. We denote by $\ag$ the bundle isomorphism $A\to P\times\hg$ induced by $\alpha$, and  by   the bundle isomorphism $V\to P\times\hg$ induced by the connection form $\theta$ of $A$.

 Let $Y$ be a differentiable manifold and $f:Y\to P$ be a smooth map. Put  $\alpha_f=f^*(\alpha)$, $\theta_f\coloneq f^*(\theta)$, $\eta_f=f^*(\eta)$.
 
 \begin{re}   $f$  is $\omega^{\alpha,g}_A$-isotropic   (see Definition \ref{isotropy-def}) if and only if for any $y\in Y$ and $u\in T_yY$ we have $g^\C(\eta_f(u),\eta_f(u))=0$, which is equivalent to the two equations
\begin{equation}\label{null}
\begin{split}
g(\alpha_f(u),\alpha_f(u)) & =g (\theta_f(u),\theta_f(u)), \\
g (\alpha_f(u),\theta_f(u))& =0.
\end{split}
\end{equation}
\end{re}

Now let $(Y,J)$ be Riemann surface. We need  explicit formulae characterising   the holomorphy condition for a map $f:Y\to P$.

\begin{re}
Let $U$, $V$ be  nowhere vanishing vector fields on $Y$ such that $V=JU$. The holomorphy condition  of $f$ becomes $f_*(V)=J^\alpha_A f_*(U)$, which, taking into account that $\eta$ is of type $(1,0)$ with respect to $J^\alpha_A$ by Lemma \ref{eta-holom}, can also be written as $\eta_f(V)=i\eta_f(U)$, or equivalently
\begin{equation}\label{HolCond}
\begin{split}
\alpha_f(V)&=-\theta_f(U),\\
\theta_f(V)&=\alpha_f(U).
 \end{split}
\end{equation}
\end{re}
Using these formulae, we obtain
\begin{re}\label{NullHolChar} Let $U$, $V$ be  nowhere vanishing vector fields on $Y$ such that $V=JU$. 
\begin{enumerate}
\item A holomorphic map $f:Y\to P$ is $\omega^{\alpha,g}_A$-isotropic if and only if
\begin{equation}\label{isotropic-for-hol}
\begin{split}
g  (\alpha_f(U),\alpha_f(U))&=g (\alpha_f(V),\alpha_f(V)),\\ g( \alpha_f(U),\alpha_f(V)) &=0.	
\end{split}
\end{equation}
\item \label{second-rem} A   holomorphic map  $f:Y\to P$ is a space-like $\omega^{\alpha,g}_A$-isotropic immersion if and only if the composition $\varphi\coloneq \pi\circ f:Y\to M$ is a space-like conformal immersion.	
\end{enumerate}
\end{re}

\begin{proof}
(1) Since $f$ is holomorphic and $Y$ is Riemann surface, it follows by Remark \ref{isotropy-cond} that $f$ is $\omega^{\alpha,g}_A$-isotropic if and only $g^\C(\eta_f(U),\eta_f(U))=0$, i.e. if and only formulae (\ref{null}) hold when both members are applied to $U$. Taking into account (\ref{HolCond}), the obtained conditions coincide with (\ref{isotropic-for-hol}).
\vspace{2mm}\\
(2)   
Let $f:Y\to P$  be a   holomorphic map.  $f$ is a space-like immersion if and only if the symmetric matrix
$$
\begin{pmatrix}
\g^\alpha_A(f_{*}(U),f_{*}(U))&\g^\alpha_A(f_{*y}(U),f_{*}(V))\\
\g^\alpha_A(f_{*}(V),f_{*y}(U))&\g^\alpha_A(f_{*}(V),f_{*}(V))
\end{pmatrix}$$
$$=\begin{pmatrix}
g(\alpha_f(U),\alpha_f(U))+g(\theta_f(U),\theta_f(U))&g(\alpha_f(U),\alpha_f(V))+g(\theta_f(U),\theta_f(V))\\
g(\alpha_f(V),\alpha_f(U))+g(\theta_f(V),\theta_f(U))&g(\alpha_f(V),\alpha_f(V))+g(\theta_f(V),\theta_f(V))
\end{pmatrix}
$$ 
$$
=\begin{pmatrix}
g(\alpha_f(U),\alpha_f(U))+g(\alpha_f(V),\alpha_f(V))&0\\
0&g(\alpha_f(V),\alpha_f(V))+g(\alpha_f(U),\alpha_f(U))
\end{pmatrix}
$$
is definite positive at any point $y\in Y$. For the equalities above we have used (\ref{formula-for-g-on-P}) and (\ref{HolCond}). Therefore $f$ is a space-like immersion if and only if at any point of $Y$ we have
\begin{equation}\label{space-like-in-P}
g(\alpha_f(U),\alpha_f(U))+g(\alpha_f(V),\alpha_f(V))>0.	
\end{equation}
Using (\ref{g-compos-pi*}) we see that 
\begin{equation*}
\begin{split}
g^\alpha_M(\varphi_{*}(U),\varphi_{*}(U))&=g(\alpha_f(U)),\alpha_f(U)),\\
g^\alpha_M(\varphi_{*}(V),\varphi_{*}(V))&=g(\alpha_f(V)),\alpha_f(V)),\\
g^\alpha_M(\varphi_{*}(U),\varphi_{*}(V))&=g(\alpha_f(U_y)),\alpha_f(V)),\\
\end{split} 
\end{equation*}
so (\ref{space-like-in-P}) becomes
\begin{equation}\label{space-like-in-P-new}
g^\alpha_M(\varphi_{*}(U),\varphi_{*}(U))+g^\alpha_M(\varphi_{*}(V),\varphi_{*}(V))>0.	
\end{equation}
The map $f$ is $\omega^{\alpha,g}_A$-isotropic if and only if (\ref{isotropic-for-hol}) holds, so if and only if
\begin{equation}\label{new-cond-isotr}
\begin{split}
g^\alpha_M(\varphi_{*}(U),\varphi_{*}(U))&=g^\alpha_M(\varphi_{*}(V),\varphi_{*}(V)),\\
g^\alpha_M(\varphi_{*}(U),\varphi_{*}(V))&=0.\\
\end{split} 
\end{equation}

Conditions (\ref{space-like-in-P-new}), (\ref{new-cond-isotr})  are obviously equivalent to the condition ``$\varphi$ is a conformal space-like immersion". 
\end{proof}

\section{The main theorem}
\label{main-th-section}

\subsection{The statement of the main theorem and its corollaries}\label{main-th-subsection}

 Using the definitions and the notations introduced in the previous sections, we can state  now our main result: 
\begin{thry}\label{main}
Let $(A,\alpha)$ be a solution of (\ref{ZSys})  with $\alpha$ admissible, $g$ be  an $\Ad_H$-invariant  non-degenerate bilinear symmetric form  on $\hg$ and let $(Y,J)$ be a Riemann surface. Endow $P$ with the complex structure $J^\alpha_A$ and the  pseudo-Hermitian metric $\g^{\alpha}_A$; endow  $M$ with the  pseudo-Riemannian  metric $ g^{\alpha}_M$.
 
\begin{enumerate}
\item 	

Let  $(Y,J)$ be a Riemann surface and $f:Y\to P$   a  space-like   $\omega^{\alpha,g}_A$-isotropic, holomorphic  immersion.  Then
\begin{enumerate} 
\item	The composition $\varphi\coloneq \pi\circ f:Y\to M$ is a space-like conformal immersion.
\item Using the   Lie algebra bundle structure $[\cdot,\cdot]_\alpha$  on $T_M$, the mean curvature vector field $\vec{\Hg}_\varphi\in A^0(Y,\varphi^*(T_M))$ of the immersion $\varphi$ is given by the formula 
\begin{equation}\label{Hg-general-formula} 
\vec{\Hg}_\varphi= \varphi^*([\cdot,\cdot]_\alpha)/\vol_\varphi\,,	
\end{equation}

where, on the right, the Lie bracket $[\cdot,\cdot]_\alpha$ is regarded as an element of $A^2(M,T_M)$ and $\varphi^*([\cdot,\cdot]_\alpha)\in A^2(Y,\varphi^*(T_M))$ stands for its pull-back via $\varphi$. 

In other words, for any point $y\in Y$ and any $\varphi^*(g^\alpha_M)$-orthonormal basis $(u,v)$ of $T_yY$ with $v=Ju$, one has  $\vec{\Hg}_\varphi(y)=[\varphi_{*y}(u),\varphi_{*y}(v)]_\alpha$. 
\end{enumerate}
\item Conversely, let $\varphi:Y\to M$ be a space-like conformal immersion  whose mean curvature vector field  is given by  formula (\ref{Hg-general-formula}). The sheaf of local holomorphic lifts of $\varphi$ (which by Remark \ref{NullHolChar} (\ref{second-rem}) coincides with the sheaf of local holomorphic lifts which are space-like $\omega^{\alpha,g}_A$-isotropic immersions) is isomorphic to the sheaf of local sections of the $H$-bundle $\varphi^*(P)\to Y$ which are parallel with respect to a flat connection $B_\varphi$ on this bundle; in particular, if $Y$ is simply connected, then  $\varphi$ admits a holomorphic lift $f:Y\to P$, which is unique up to  right translation by an element of $H$, and which is a space-like $\omega^{\alpha,g}_A$-isotropic immersion.
\end{enumerate}	
\end{thry}

  Comparing our explicit formula for the the mean curvature vector field $\vec{\Hg}_\varphi$ with the known formula for the curvature of   $\nabla^\alpha_A$ given by Theorem \ref{mainZe} (2) (which by Corollary \ref{coro-LC} coincides with the Riemannian curvature of $g^\alpha_M$), we see that both are determined via very simple formulae by the Lie algebra structure $[\cdot,\cdot]_\alpha$ on tangent bundle $T_M$. In particular we obtain the following result which related $\vec{\Hg}_\varphi$ to the pull back $\varphi^*(R)$ of the Riemann curvature of $g^\alpha_M$.

\begin{co}\label{main-coro}
In the conditions of Theorem \ref{main}	, let $\varphi:Y\to M$ be the space-like conformal immersion associated with a space-like $\omega^{\alpha,g}_A$-isotropic holomorphic immersion $f:Y\to P$.  Let $R\in  A^2(\so(T_M))$ be the Riemannian curvature of $(M,  g^\alpha_M)$, and $K$  its  sectional curvature.  
\begin{enumerate}
\item Put $[\cdot,\cdot]_\alpha^\varphi\coloneq \varphi^*([\cdot,\cdot]_\alpha)\in A^2(Y,\varphi^*(T_M))$. The pull-back 
$$\varphi^*(R)\in A^2(Y,\so(\varphi^*(T_M)))$$
 is related to the mean curvature vector field $\vec{\Hg}_\varphi\in A^0(Y,\varphi^*(T_M))$ of $\varphi$  by the equality  
\begin{equation}
\label{HintermsofR}[\vec{\Hg}_\varphi,\cdot]_\alpha^\varphi\vol_\varphi=\varphi^*(R). 
\end{equation}	
\item  For any $y\in Y$ one has
$$\langle \vec{\Hg}_\varphi(y), \vec{\Hg}_\varphi(y)\rangle_{g^\alpha_M} = - K(\im(\varphi_{*y})).$$
\end{enumerate}
\end{co}

Note that if the center of $\hg$ vanishes (in particular if $\hg$ is semi-simple) the endomorphism  $[\vec{\Hg}_\varphi,\cdot]^\varphi_\alpha$ of $\varphi^*(T_M)$ determines $\vec{\Hg}_\varphi$. It follows that in this case formula (\ref{HintermsofR}) expresses the mean curvature vector field $\vec{\Hg}_\varphi$ of $\varphi$  in terms of  the pull-back $\varphi^*(R)$ of the  Riemannian curvature of $(M,  g^\alpha_M)$.
\\

The following corollary shows that the   Bryant  and Aiyama-Akutagawa-Fujimori-Lee correspondences can be recovered as special cases of  
 Theorem \ref{main}. Note that in these correspondences  the  condition on the considered conformal immersions $\varphi:Y\to M$ makes use of the scalar mean curvature, not of the mean curvature vector field.    
 
Let $(\hg,g)$ be  simple 3-dimensional real Lie algebra  endowed with non-degenerate $\ad$-invariant (see \cite[section 3.6]{Bou}) symmetric bilinear form $g:\hg\times\hg\to \R$ which admits space-like planes, i.e. whose signature is either (3,0) or (2,1). By the classification of simple Lie algebras, it follows that $\hg$ is isomorphic to either $\su(2)$ or $\sl(2,\R)$. Since the complexification $\hg\otimes\C$ is also simple, it follows (see \cite[(18)(a) p. 131]{Bou}) that $g=cB_\hg$, where $B_\hg$ is the Killing form of $\hg$ and $c\in \R^*$. The condition on the signature shows that $c<0$ when $\hg\simeq \su(2)$ and $c>0$ when $\hg\simeq \sl(2,\R)$. 

In the latter case, let $(u,v)\in \hg\times\hg$ be a linearly independent pair spanning a space-like plane. It is easy to check that then $B_\hg([u,v],[u,v])<0$, in particular $[u,v]\not\in \langle u,v\rangle$, so $(u,v,[u,v])$ is a basis of $\hg$. 

This shows that any simple 3-dimensional real Lie algebra $\hg$ comes with a canonical orientation defined by any  basis of the form $(u,v,[u,v])$, where $(u,v)$ is either an arbitrary linearly independent pair  (when $\hg\simeq \su(2)$) or a linearly independent pair spanning  a space-like plane  with respect to $B_\hg$ (when $\hg\simeq \sl(2,\R)$).

 This canonical orientation is invariant under any automorphism of $\hg$. It follows that, if $\hg$ is the Lie algebra of a (not necessarily connected) Lie group $H$, this canonical orientation is $\Ad_H$-invariant. With this remark we can state:
 \begin{co}\label{scalar-mean-curv-coro}
 In the conditions of Theorem \ref{main} suppose that $\hg$ is a simple 3-dimensional real Lie algebra. Endow $M$ with the orientation induced by the canonical orientation of $\hg$ via the isomorphism $j^\alpha_M:T_M\textmap{\simeq} \Ad(P)$.  	The scalar mean curvature of the space-like conformal immersion $\varphi$ associated with a space-like $\omega^{\alpha,g}_A$-isotropic holomorphic immersion $f:Y\to P$ is positive and is given by the formula
\begin{equation}\label{scalar-curvature-formula} \Hg_\varphi(y)=\sqrt{|K(\im(\varphi_{*y}))|}.
\end{equation}
 Conversely, let $\varphi:Y\to M$ be a space-like conformal immersion  whose scalar mean curvature is given by this formula. The sheaf of local holomorphic lifts of $\varphi$ is isomorphic to the sheaf of local sections of the $H$-bundle $\varphi^*(P)\to Y$ which are parallel with respect to a flat connection $B_\varphi$ on this bundle; in particular, if $Y$ is simply connected, then  $\varphi$ admits a holomorphic lift $f:Y\to P$, which is unique up to  right translation by an element of $H$, and which is a space-like $\omega^{\alpha,g}_A$-isotropic immersion.
 \end{co}
\begin{proof}
The scalar  mean curvature $\Hg_\varphi$ of $\varphi$   is determined by the formula 
$$\vec\Hg_\varphi=\Hg_\varphi \eta,$$
 where $\eta$ is the Gauss vector field of the space-like immersion $\varphi$ associated with the canonical orientation of $Y$ and the specified orientation of $M$; $\eta$ is determined by the conditions:
 \begin{enumerate}
 \item For any $y\in Y$,  $\eta_y$ is orthogonal on $\varphi_{*y}(T_yY)$ with respect to $g^\alpha_M$.
 \item For any $y\in Y$ and any basis $(u,v)$ of $T_yY $ which is compatible with the canonical orientation of $Y$,  the triple $(\varphi_{*y}(u),\varphi_{*y}(v),\eta_y)$ is compatible with the specified orientation of $M$.
 \item $|g^\alpha_M(\eta,\eta)|=1$.
 \end{enumerate}
 Note that in the non-Riemannian case (when $\hg\simeq \sl(2,\R)$), we have $g^\alpha_M(\eta,\eta)=-1$. Taking into account the definition of the specified orientation of $M$ and making use again of Corollary \ref{coro-LC}, we see that, choosing a basis $(u,v)$ of $T_yY $ as in Theorem \ref{main} (1) (b), we have 
 \begin{align*}
\eta(y)=&\frac{1}{\sqrt{\big|g^\alpha_M\big([\varphi_{*y}(u),\varphi_{*y}(v)]_\alpha,[\varphi_{*y}(u),\varphi_{*y}(v)]_\alpha\big)\big| }}
 [\varphi_{*y}(u),\varphi_{*y}(v)]_\alpha\\
 =&\frac{1}{\sqrt{\big|K(\im(\varphi_{*y}))\big| }}[\varphi_{*y}(u),\varphi_{*y}(v)]_\alpha.
  \end{align*}
Therefore (\ref{scalar-curvature-formula}) is equivalent to (\ref{Hg-general-formula}), so the claim follows from   Theorem \ref{main}.
 
\end{proof}

Note that in Example \ref{simple-group-ex} of the next section we will describe explicitly the 4-tuples $(P\textmap{\pi}M,\alpha,A,g)$ which give the Bryant  and Aiyama-Akutagawa-Fujimori-Lee correspondences.

\subsection{The proof of the main theorem}\label{ProofSection}

In this section we use the assumptions and the notations introduced in section \ref{first-prop-sect}.

\begin{pr}
\label{dthetaf:pr}
Let $f:Y\to P$	be a holomorphic map. Then 
\begin{equation}\label{dthetaf}
d\theta_f=0.	
\end{equation}

\end{pr}
\begin{proof} 

The problem is local, so we may assume that nowhere vanishing vector fields $U$, $V$ with $V=JU$ exist. By (\ref{HolCond}) we have 
$$\frac{1}{2}[\alpha_f\wedge\alpha_f](U,V)=[\alpha_f(U),\alpha_f(V)]=[\theta_f(V),-\theta_f(U)]=\frac{1}{2}[\theta_f\wedge\theta_f](U,V).
$$
Therefore
\begin{equation}\label{alpha-alpha=theta-theta}
[\alpha_f\wedge\alpha_f]=[\theta_f\wedge\theta_f].	
\end{equation}
The claim follows taking  the  pull-back via $f$ of the structure equation 
$$\Omega_A=d\theta+\frac{1}{2}[\theta\wedge\theta],
$$
(see section \ref{connections-section}), and  of the second integrability  condition $\Omega_A=\frac{1}{2}[\alpha\wedge\alpha]$ for $J^\alpha_A$.
\end{proof}

Let $(Y,J)$ be a Riemann surface and $\varphi:Y\to M$ a smooth map. Consider the cartesian diagram
\begin{equation}\label{Bundle-Map-Phi}
\begin{tikzcd}
P_\varphi\coloneq Y\times_\varphi P \ar[d, "\pi_\varphi"'] \ar[r,"\Phi"]  & P\ar[d,"\pi"]\\
Y\ar[r,"\varphi"] & M
\end{tikzcd},
\end{equation}
where $P_\varphi$ the pull-back of $P$ via $\varphi$ and $\Phi$ is defined by the second projection; it is a $\varphi$-covering morphism of principal $H$-bundles. 

We have an obvious bundle isomorphism
$$
j^\alpha_\varphi\coloneq\varphi^*(j^\alpha_M) :\varphi^*(T_M)\textmap{\simeq} \varphi^*(\Ad(P))=\Ad(P_\varphi),
$$
induced by $j^\alpha_M$ via $\varphi$ given explicitly by
$$
T_{\varphi(y)}M\ni w\mapsto j^\alpha_\varphi(w)=j^\alpha_M(w)=\tilde\alpha(w)\in \Ad(P)_{\varphi(y)}=\Ad(P_\varphi)_y.
$$
Therefore we also obtain canonical isomorphisms 
$$
j^\alpha_\varphi:A^k(Y,\varphi^*(T_M))\textmap{\simeq}A^k(Y,\Ad(P_\varphi))$$
which will be denoted by the same symbol $j^\alpha_\varphi$ to save on notations.

Put $\alpha_\varphi\coloneq\Phi^*(\alpha)\in A^1_\Ad(P_\varphi,\hg)$, and let $\tilde\alpha_\varphi=\varphi^*(\tilde\alpha)\in  A^1(Y,\Ad(P_\varphi))$ be the associated $\Ad(P_\varphi)$-valued 1-form. The composition $j^\alpha_\varphi\circ \varphi_*:T_Y\to \Ad(P_\varphi)$ is given explicitly by 
\begin{equation}\label{j-alpha-phi-varphi*}
T_yY\ni v\mapsto  j^\alpha_\varphi(\varphi_*(v))=\tilde\alpha(\varphi_*(v))=\tilde\alpha_\varphi(v).
\end{equation}
Let $A_\varphi=\varphi^*(A)\in {\cal A}(P_\varphi)$ be  the pull-back of $A$ via $\varphi$. Its connection form is 
\begin{equation}\label{theta-phi}
\theta_\varphi\coloneq\Phi^*(\theta).	
\end{equation}
In our proofs we will need the associated linear connection  $\nabla^\Ad_{A_\varphi}$    on  $\Ad(P_\varphi)$ and  the associated   de Rham operator $d_{A_\varphi}$ (see section \ref{connections-section}). 

 \begin{pr}\label{FormulaHLm}
Let $(Y,J)$ be a Riemann surface,  $\varphi:Y\to M$ a space-like conformal immersion, and let $\vec{\Hg}_\varphi\in A^0(Y,\varphi^*(T_M))$ be its mean curvature vector field. With the notations introduced above, we have in  $A^2(Y,\Ad(P_\varphi))$ the equality
\begin{equation}\label{FormulaH}
j^\alpha_\varphi(\vec{\Hg}_\varphi)\vol_\varphi= -\frac{1}{2}d_{A_\varphi}(J\tilde \alpha_\varphi).
\end{equation}
\end{pr}
\begin{proof} We introduce the following notations:
\begin{itemize}
\item We denote by $g_Y$  the  Riemannian metric on $Y$ induced by $g^\alpha_M$ via $\varphi_*$, and by $g_\varphi$ be the pseudo-Riemannian metric on the pull-back bundle $\varphi^*(T_M)$  induced by $g^\alpha_M$ via $\varphi$. 
\item  We denote by $\nabla:=\varphi^*(\nabla^\alpha_A)$ the pull-back of the Levi-Civita connection $\nabla^\alpha_A$ of $(M,g^\alpha_M)$ to   $\varphi^*(T_M)$. Here we have used Corollary \ref{coro-LC}.
\end{itemize}

Note that the bundle embedding $\varphi_*:T_Y\to \varphi^*(T_M)$ is an isometric embedding with respect to the pair $(g_Y,g_\varphi)$. Since $\nabla^\alpha_A$ is compatible with $g^\alpha_M$,  it follows that $\nabla=\varphi^*(\nabla^\alpha_A)$ is compatible with $g_\varphi$. 

Our problem is local, so we may assume that the tangent bundle $T_Y$ is trivial. Since $\varphi$ is conformal, we can find $U$, $V$ be vector fields on $Y$ which are orthonormal at any point with respect to  $g_Y$, and such that $V=JU$. Let $U^\varphi$, $V^\varphi$ be the corresponding sections  of   $\varphi^*(T_M)$; they are defined by the maps
$$
x\mapsto \varphi_*(U_x),\ x\mapsto \varphi_*(V_x).
$$
Using formula (\ref{expl-formula-for-Hg}) and Remark \ref{Hg-for-space-like}, we see that,  with these notations, we have   
$$\vec{\Hg}_\varphi=\frac{1}{2}\big(\nabla_U U^\varphi+\nabla_V V^\varphi\big)^\bot,$$
where $\bot$ stands for the projection on the $g_\varphi$-orthogonal complement of the subbundle $\varphi_*(T_Y)\subset \varphi^*(T_M)$. Choose $y\in Y$ and assume that $[U,V]_y=0$. This can be achieved by using vector fields which are parallel at $y$ and by the fact that the Levi-Civita connection is torsion free. We claim that, with this choice, $\nabla_{U_y} U^\varphi$, $\nabla_{V_y} V^\varphi$ are already orthogonal to the plane $\varphi_*(T_yY)$. Indeed, taking into account that $\nabla$ is compatible with   $g_\varphi$  we get
\begin{equation*}
\begin{split}
g_\varphi(\nabla_{U_y} U^\varphi,U^\varphi_y)&=0,\\
 g_\varphi(\nabla_{V_y} U^\varphi,U^\varphi_y)&=0,	\\
  g_\varphi(\nabla_{U_y} U^\varphi,V^\varphi_y)&=-g_\varphi( U^\varphi_y,\nabla_{U_y}V^\varphi).
\end{split}	
\end{equation*}
On the other hand, since $\nabla^\alpha_A$ is  torsion free, we have
\begin{equation}\label{using-torsion-free}
\nabla_{U} V^\varphi-\nabla_{V} U^\varphi=[U,V]^\varphi.	
\end{equation}
This can be obtained in the usual way by identifying $Y$ locally around $y$ (via $\varphi$) with a submanifold of $M$ and extending  $U$, $V$ to vector fields on an open neighbourhood of $y$ in $M$ (see for instance \cite[section 2.2]{Ch1}). The condition $[U,V]_y=0$ now gives
\begin{equation}\label{after-using-torsion-free}
\nabla_{U_y} V^\varphi=\nabla_{V_y} U^\varphi,	
\end{equation}
 which shows that  
$$g_\varphi(\nabla_{U_y}U^\varphi,V_y^\varphi)=-g_\varphi( U^\varphi_y,\nabla_{U_y}V^\varphi)=-g_\varphi( U^\varphi_y,\nabla_{V_y} U^\varphi)=0,$$
so $\nabla_{U_y} U^\varphi$ is orthogonal on both $U_y$ and $V_y$. A similar argument applies to  $\nabla_{V_y} V^\varphi$.
 Therefore
\begin{equation}\label{formula-for-Hg}
\vec{\Hg}_\varphi(y)=\frac{1}{2}\big(\nabla_{U_y} U^\varphi+\nabla_{V_y} V^\varphi\big).
\end{equation}
The sections of $\Ad(P_\varphi)$ which correspond  to the sections 
$U^\varphi,\ V^\varphi\in \Gamma(Y,\varphi^*(T_M))
$
 via the bundle isomorphism $ j^\alpha_\varphi:\varphi^*(T_M)\textmap{\simeq} \varphi^*(\Ad(P))$ are $\tilde\alpha_\varphi(U)$, respectively $\tilde\alpha_\varphi(V)$.

 Since   $\nabla^\alpha_A=(j^\alpha_M)^*(\nabla^\Ad_A)$, we obtain (taking pull-back via $\varphi$) $\nabla=(j^\alpha_\varphi)^*(\nabla^\Ad_{A_\varphi})$. Therefore  formula (\ref{formula-for-Hg}) gives:
\begin{equation}\label{phi*(j)(H)}
 j^\alpha_\varphi(\vec{\Hg}_\varphi)(y)=\frac{1}{2}\big((\nabla^\Ad_{A_\varphi})_{U_y}  \tilde\alpha_\varphi(U)+\ (\nabla^\Ad_{A_\varphi})_{V_y} \tilde\alpha_\varphi(V)\big).	
\end{equation}
On the other hand
\begin{equation*}
\begin{split}
d_{A_\varphi}(J\tilde\alpha_\varphi)(U_y,V_y)&= (\nabla^\Ad_{A_\varphi})_{U_y}(J\tilde\alpha_\varphi(V))- (\nabla^\Ad_{A_\varphi})_{V_y}(J\tilde\alpha_\varphi(U))-J\tilde\alpha_\varphi([U,V]_y)\\
 &=- (\nabla^\Ad_{A_\varphi})_{U_y}(\tilde\alpha_\varphi(U))-(\nabla^\Ad_{A_\varphi})_{V_y}(\tilde\alpha_\varphi(V))-J\tilde\alpha_\varphi([U,V]_y).	
\end{split}	
\end{equation*}
The third term vanishes at $y$, so, by (\ref{phi*(j)(H)}),  $j^\alpha_\varphi(\vec{\Hg}_\varphi)(y)=-\frac{1}{2}d_{A_\varphi}(J\alpha_\varphi)(U,V)(y)$, which proves the claim.
 
\end{proof}

With these preparations we can now give the proof of Theorem \ref{main}:

\begin{proof}

(1) Let $f:Y\to P$ be a space-like $\omega^{\alpha,g}_A$-isotropic, holomorphic immersion.
\vspace{-1mm}\\
(1)(a) By  Remark \ref{NullHolChar} (\ref{second-rem}), the composition  $\varphi:=\pi\circ f$ is a space-like conformal immersion. 
\vspace{2mm}\\ 
(1)(b) The map $\tau_f:Y\to P_\varphi$ given by $y\mapsto (y,f(y))$ is a section of $P_\varphi$, so (see section \ref{connections-section}) it induces a map
$$
l^{\tau_f}_\Ad:\Ad(P_\varphi)\to \hg
$$ 
(whose restrictions to the  fibres are Lie algebra isomorphisms) and vector space isomorphisms 
$$
l^{\tau_f}_\Ad: A^k(Y,\Ad(P_\varphi))\textmap{\simeq}  A^k(Y,\hg)	
$$
denoted by the same symbol (see section \ref{connections-section} for the general formalism). We have
\begin{equation}\label{tau-f*(theta-phi)}
\tau_f^*(\theta_\varphi)=\tau_f^*(\Phi^*(\theta))=f^*(\theta)=\theta_f.	
\end{equation}
On the other hand,
\begin{equation}\label{tau-f*(alpha-phi)}
l^{\tau_f}_\Ad(J\tilde \alpha_\varphi)=Jl^{\tau_f}_\Ad(\tilde\alpha_\varphi)=J\tau_f^*(\alpha_\varphi)=J\tau_f^*(\Phi^*(\alpha))=Jf^*(\alpha)=J\alpha_f=-\theta_f.
\end{equation}
For the first equality we that the correspondance $\lambda\mapsto J\lambda$ defined in section \ref{connections-section} obviously commutes with left compositions by vector bundle isomorphisms, for the second  equality we have used (\ref{obv-id}), and for the last equality we have used formulae (\ref{HolCond}).
Using formulae (\ref{dA-with-tau}), (\ref{tau-f*(alpha-phi)}) and Proposition \ref{dthetaf:pr}, it follows
$$
l^{\tau_f}_\Ad(d_{A_\varphi}(J\tilde\alpha_\varphi))=(d+[\theta_f\wedge .])(l^{\tau_f}_\Ad(J\tilde\alpha_\varphi))=-d\theta_f-\theta_f\wedge\theta_f=-\theta_f\wedge\theta_f.
$$
 Proposition \ref{FormulaHLm} now gives
\begin{equation}\label{l-tau-f-j-alpha-phi(H)}
l^{\tau_f}_\Ad(j^\alpha_\varphi(\vec{\Hg}_\varphi)\vol_\varphi)=-\frac{1}{2}l^{\tau_f}_\Ad(d_{A_\varphi}(J\alpha_\varphi))=\frac{1}{2}\theta_f\wedge\theta_f.
\end{equation}

Recall the notation $[\cdot,\cdot]_\alpha^\varphi\coloneq \varphi^*([\cdot,\cdot]_\alpha)\in A^2(Y,\varphi^*(T_M))$ introduced in section \ref{main-th-subsection}. Since the bundle isomorphism $j^\alpha_\varphi:\varphi^*(T_M)\to \Ad(P_\varphi)$  induces fibrewise Lie algebra isomorphisms  and using (\ref{j-alpha-phi-varphi*}),  we have 
\begin{equation*} 
j^\alpha_\varphi([\cdot,\cdot]^\varphi_\alpha)(u,v)=j^\alpha_\varphi([\varphi_*(u),\varphi_*(u)]_\alpha)=[j^\alpha_\varphi(\varphi_*(u)),j^\alpha_\varphi(\varphi_*(v))]=[\tilde\alpha_\varphi(u),\tilde\alpha_\varphi(v)]
\end{equation*}
for any $y\in Y$, $u$, $v\in T_yY$.  Therefore
\begin{equation}\label{new-eqq}
j^\alpha_\varphi([\cdot,\cdot]^\varphi_\alpha)=\frac{1}{2}[\tilde\alpha_\varphi\wedge \tilde\alpha_\varphi]
\end{equation}
in $A^2(Y,\Ad(P_\varphi))$, where, on the right, $[\cdot,\cdot]$ stands for the canonical fibrewise Lie algebra bracket on the bundle $\Ad(P_\varphi)$. 
Since $l^{\tau_f}_\Ad$ also  induces  Lie algebra  isomorphisms on the fibres, we infer
\begin{equation}\label{l-tau-Ad([])}
l^{\tau_f}_\Ad(j^\alpha_\varphi([\cdot,\cdot]^\varphi_\alpha)=\frac{1}{2}l^{\tau_f}_\Ad([\tilde\alpha_\varphi\wedge \tilde\alpha_\varphi])=\frac{1}{2}[l^{\tau_f}_\Ad(\tilde\alpha_\varphi)\wedge l^{\tau_f}_\Ad(\tilde\alpha_\varphi)].
\end{equation}
Now note that $\tilde\alpha_\varphi=\widetilde{\Phi^*(\alpha)}$, so the general formula (\ref{obv-id}) gives 
$$l^{\tau_f}_\Ad(\tilde\alpha_\varphi)=\tau_f^*(\Phi^*(\alpha))=f^*(\alpha)=\alpha_f,$$
and (\ref{l-tau-Ad([])}) becomes
\begin{equation}\label{l-tau-Ad([])-new}
l^{\tau_f}_\Ad(j^\alpha_\varphi([\cdot,\cdot]^\varphi_\alpha)=\frac{1}{2}[\alpha_f\wedge\alpha_f].	
\end{equation}
Using (\ref{l-tau-f-j-alpha-phi(H)}), (\ref{l-tau-Ad([])-new}) and the known equality (\ref{alpha-alpha=theta-theta}), we obtain
$$l^{\tau_f}_\Ad(j^\alpha_\varphi(\vec{\Hg}_\varphi)\vol_\varphi)=l^{\tau_f}_\Ad(j^\alpha_\varphi([\cdot,\cdot]^\varphi_\alpha),$$
in $A^2(Y,\hg)$, so $\vec{\Hg}_\varphi\vol_\varphi=[\cdot,\cdot]^\varphi_\alpha$ in $A^2(Y,\varphi^*(T_M))$ as claimed. \\

(2) Let $\varphi:Y\to M$ be a space-like conformal immersion whose mean curvature vector field is $\vec{\Hg}_\varphi=\varphi^*([\cdot,\cdot])/\vol_\varphi$. The data of a (local) lift  of $\varphi$ to $P$ is equivalent to the data of a (local) section $\tau$ in the pull-back bundle 
$$\pi_\varphi:P_\varphi\coloneq Y\times_\varphi P\to Y.$$
Using the notations introduced in Proposition \ref{FormulaHLm}, formulae (\ref{HolCond}) show that a (local) lift $f$ of $\varphi$ is holomorphic if and only if the associated section $\tau_f=(\id_Y,f)$ of $P_\varphi$ is parallel with respect to the connection $B_\varphi\coloneq A_\varphi+ J\alpha_\varphi$. By formula (\ref{Omega(A+alpha)}) the curvature form $F_{B_\varphi}\coloneq\tilde\Omega_{B_\varphi}\in A^2(Y,\Ad(P_\varphi))$  of this connection is
$$F_{B_\varphi}=\varphi^*(F_A)+d_{A_\varphi} (J\tilde\alpha_\varphi)+\frac{1}{2} [(J\tilde\alpha_\varphi)\wedge (J\tilde\alpha_\varphi)]=d_{A_\varphi} (J\tilde\alpha_\varphi)+[\tilde\alpha_\varphi\wedge\tilde \alpha_\varphi].
$$
For the second equality, we  used the obvious equality $[(J\tilde\alpha_\varphi)\wedge (J\tilde\alpha_\varphi)]=[\tilde\alpha_\varphi\wedge\alpha_\varphi]$, and the integrability condition $\Omega_A=\frac{1}{2}[\alpha\wedge\alpha]$, which yields $F_A=\frac{1}{2}[\tilde\alpha\wedge\tilde\alpha]$.  Using Proposition \ref{FormulaHLm} and the assumption $\vec{\Hg}_\varphi=\varphi^*([\cdot,\cdot])/\vol_\varphi$ gives $F_{B_\varphi}=0$, so ${B_\varphi}$ is flat, which proves the claim. 

 \end{proof}
 
 \subsection{Weierstrass representations and simple Weierstrass representations}
 \label{Weierstrass-repr-section}
 
In this section we use the notations introduced in section \ref{main-th-subsection} and we assume the hypothesis   of Theorem \ref{main}. 
 
We define:
 
 \begin{dt} Let $\varphi:Y\to M$ be a space-like conformal immersion whose mean curvature vector field $\vec{\Hg}_\varphi$ is given by formula (\ref{Hg-general-formula}).
 A Weierstrass  representation of $\varphi$ is commutative diagram
 $$
 \begin{tikzcd}
 \tilde Y \ar[r, "\tilde f"] \ar[d, "c"'] & P\ar[d, "\pi"]\\
 Y\ar[r, "\varphi"]& M	
 \end{tikzcd}, \eqno{(W)}
 $$
 where $c$ is a covering map of Riemann surfaces and $\tilde f$ is holomorphic.
 
 A simple Weierstrass representation of $\varphi$ is a holomorphic lift $f:Y\to P$ of $\varphi$.
 \end{dt}

The terminology "simple Weierstrass representation" used here was chosen to agree with  the terminology introduced in \cite[Section 6.1]{BDHH} in the special framework of   bundles of the the form $\Gamma\backslash \SL(2,\C)\to \Gamma\backslash \SL(2,\C)/\SU(2)$ associated with a lift $\Gamma$ of a torsion free Kleinian group $\Gamma_0\subset \PGL(2,\C)$ (see section {\ref{Bryant-for-Gamma-section} in this article}). Note that

\begin{re}
For any (simple) Weierstrass representation (W) of $\varphi$, the map $\tilde f:\tilde Y\to P$ (respectively $f:Y\to P$) is not only holomorphic, but also a space-like 	$\omega^{\alpha,g}_A$-isotropic immersion.
\end{re}

Taking into account Theorem \ref{main} and its proof we obtain:
\begin{pr}\label{existence-Weierstrass}
Let $\varphi:Y\to M$ be a space-like conformal immersion whose mean curvature vector field $\vec{\Hg}_\varphi$ is given by formula (\ref{Hg-general-formula}). Then
\begin{enumerate} 
\item $\varphi$ admits a	 Weierstrass  representation (W), where $c$ is a normal covering and $\tilde f$ is a $\varphi$-covering morphism of principal bundles which is equivariant with respect to a group monomorphism $\Aut_Y(\tilde Y)\hookrightarrow H$.
\item $\varphi$ admits a simple Weierstrass representation if and only if the holonomy of the connection $B_\varphi$ on the bundle $P_\varphi\coloneq \varphi^*(P)$  is trivial (see the proof of Theorem \ref{main} (2)).
\end{enumerate}
\end{pr}
\begin{proof}
(1) Fix $y_0\in Y$,  $\xi_0\in P_{\varphi(x_0)}$ and $h_{y_0}^{B_\varphi}:\pi_1(Y,y_0)\to \Aut_H(P_{\varphi(x_0)})$  the holonomy morphism associated with the flat connection $B_\varphi$. Let $\tilde Y\subset P_\varphi$ be the maximal integral manifold  of $B_\varphi$ (regarded as involutive involution on $P_\varphi$) which passes through $(y_0,\xi_0)$ \cite[Proposition 1.2, p. 10]{KN}.  By the general theory of flat connections, the restriction $\pi_\varphi|_{\tilde Y}:\tilde Y\to Y$ is a normal covering  which corresponds to the normal subgroup $\ker(h_{y_0}^{B_\varphi})\subset \pi_1(Y,y_0)$ and whose fibre  over $y_0$ coincides with $\im(h_{y_0}^{B_\varphi})(\xi_0)$. The deck transformation group $\Aut_Y(\tilde Y)$ is identified with the stabiliser of the subset $\tilde Y\subset P_\varphi$  with respect to  the canonical right $H$-action, and is isomorphic to $\im(h_{y_0}^{B_\varphi})$.

Endowing $\tilde Y$ with the holomorphic structure induced from $Y$, the proof of Theorem \ref{main} (2) shows that the restriction $\tilde f\coloneq \Phi|_{\tilde Y}$ is holomorphic.
\end{proof}

Note that the second statement of Proposition \ref{existence-Weierstrass} is a generalisation of \cite[Theorem 6.3 p. 31]{BDHH}.

 \section{Real forms of complex Lie groups} \label{RealForms-section}

Theorem \ref{main} yields a generalisation of Bryant's theorem in the framework of principal bundles. The input data intervening in this theorem is a 4-tuple 
 $$(P\textmap{\pi}M,\alpha,A,g),$$
  where $P\textmap{\pi}M$ is a principal $H$-bundle, $(\alpha,A)\in A^1_\Ad(P,\hg)\times{\cal A}(P)$ is a solution of the differential system (\ref{ZSys}) with 
 $\alpha$ admissible, and $g$ an $\Ad$-invariant, non-degenerate, bilinear form on $\hg$.
 
We will see below that any real form $H$ of a complex Lie group $G$ yields in a canonical way a triple $(P\textmap{\pi}M,\alpha,A)$ as above. Therefore, as a consequence of our general Theorem \ref{main}, we obtain a Bryant type theorem which concerns conformal space-like immersions $Y\to G/H$ from a Riemann surface into any homogeneous pseudo-Riemannian space of the form $G/H$. 

\subsection{The Bryant theorem for real forms of complex Lie groups}\label{RealForms-subsection}
  
Let $G$ be a connected complex Lie group, $J_G$ its canonical almost complex structure and let $H$ a real form of $G$, i.e. a closed subgroup whose Lie algebra $\hg$ is a real form of the Lie algebra $\g$ of $G$. We don't require $H$ to be connected. Note that $i\hg$ is an $\Ad_H$-invariant complement of $\hg$ in $\g$. Indeed, for any $h\in H$, the inner automorphism $\iota_h\in\Aut(G)$ is holomorphic, so its tangent map $\Ad_h=(\iota_h)_{*e}$ at the unit $e\in G$ is $\C$-linear.  This shows in particular that $(G,H)$ is a reductive pair (see \cite[Example 4, p. 165]{SaWa}).

Let $\eta\in A^1(G,\g)$ be the canonical left invariant $\g$-valued 1 form on $G$ (voir \cite[p.  41]{KN}). This form is defined by
$$
\eta(v)=l_{\gamma^{-1}*}(v) \hbox{ for any } \gamma\in G \hbox{ and for any }v\in T_\gamma G.
$$
One can write formally $\eta=\gamma^{-1}d\gamma$. It is well-known that $\eta\circ J_G=i\eta$, i.e. $\eta$ is of type (1,0), and, regarded as a fibrewise $\C$-linear map on the holomorphic tangent bundle ${\cal T}_G$, it is holomorphic. 

Using the direct sum decomposition $\g=\hg\oplus i\hg$, we can decompose the form $\eta$ as
$$
\eta=\theta-i\alpha,
$$
where $\theta$ is the real part and $-\alpha$ is the imaginary part of $\eta$ with respect to the real structure defined by the real form $\hg$ of $\g$.

We will regard the canonical submersion $\pi_H:G\to G/H\eqcolon M$ as a principal $H$-bundle. 

\begin{re}\label{alpha-A-of-real-form}
\begin{enumerate}
\item The forms $\alpha$, $\theta\in A^1(G,\hg)$ are left invariant. 
\item $\alpha$ is an admissible 	tensorial form of type $\Ad_H$ on $G$. 
\item $\theta$ is the connection form of a left invariant connection $A$ on $G$.
\item The almost complex structure $J^\alpha_A$ on $G$ coincides with the integrable almost complex structure $J_G$, in particular, by Theorem \ref{mainZe}, the pair $(\alpha,A)$ is a solution of the differential system (\ref{ZSys}).
\end{enumerate}
  	
\end{re}

 Let $g$ be a non-degenerate,  $\Ad_{H}$-invariant,   symmetric, $\R$-bilinear form on $\hg$ and let $g^\bot$ be the induced bilinear form on $i\hg$ via the obvious isomorphism $\hg\textmap{\simeq}i\hg$.
 The associated pseudo-Riemannian metric $g^\alpha_M$ in the sense of section \ref{section-for-g-alpha-M-omega} is just the unique $G$-invariant metric $g_M$ on $M$ which agrees with   $g^\bot$ at $[e]_H$ via the obvious isomorphism $i\hg\textmap{\simeq}T_{[e]_H}M$. The pseudo-Hermitian metric $\g^\alpha_A$  is the unique left invariant pseudo-Riemannian metric $\g$ on $G$ which coincides with  $g\oplus g^\bot$ at $e$. Finally, the holomorphic form $\omega^{\alpha,g}_A$   introduced in section  \ref{section-for-g-alpha-M-omega} is the unique left invariant symmetric, $\C$-bilinear form $\omega^g_H$ on $G$ which coincides at $e$ with the $\C$-bilinear extension $g^\C$ of $g$. Note that, since $g$ is $\Ad_H$-invariant, it follows that $g^\C$ is also $\Ad_H$-invariant, so also $\Ad_G$-invariant. Therefore $\omega^g_H$ is a   holomorphic, bi-invariant, non-degenerate, symmetric form on $G$. 
 
 The Lie bracket $[\cdot,\cdot]_\alpha:T_M\times_M T_M\to T_M$ induced by $\alpha$ on $T_M$ depends only on $H$, so it will be denoted by $[\cdot,\cdot]_H$.
 
As a special case of our main Theorem \ref{main}, we obtain:
 
 \begin{thry} \label{main-real-forms}  Let $H$ be a real form of $G$,  $g$ an $\Ad_H$-invariant  non-degenerate bilinear symmetric form   on $\hg$ and $(Y,J)$ a Riemann surface. Endow $G$ with the  pseudo-Hermitian metric $\g=g\oplus g^\bot$ and the holomorphic symmetric $\C$-bilinear form $\omega^g_H$, and endow $M\coloneq G/H$ with the pseudo-Riemannian  metric $g_M$. 
 
\begin{enumerate} 
\item 	Let $f:Y\to G$  a space-like $\omega^g_H$-isotropic holomorphic immersion.   
\begin{enumerate} 
\item	The composition $\varphi\coloneq \pi\circ f:Y\to M$ is a space-like conformal immersion.
\item Using the canonical Lie algebra bundle structure $[\cdot,\cdot]_H$ on $T_M$, the mean curvature vector field  of the immersion $\varphi$ is given by the formula 
\begin{equation} \label{Hg-RealForms-formula}
\vec{\Hg}_\varphi=\varphi^*([\cdot,\cdot]_H)/\vol_\varphi\,,
\end{equation}

where, on the right hand side, the Lie bracket $[\cdot,\cdot]_H$ is regarded as an element of $A^2(M,T_M)$ and $\varphi^*([\cdot,\cdot]_H)\in A^2(Y,\varphi^*(T_M))$ stands for its pull-back via $\varphi$.
In other words, for a point $y\in Y$ one has  $\vec{\Hg}_\varphi(y)=[\varphi_*u,\varphi_*v]_H$, where $(u,v)$ is a $\varphi^*(g_M)$-orthonormal basis of $T_yY$ such that $v=Ju$. 
\end{enumerate}

\item

Conversely, let $\varphi:Y\to M$ be a space-like conformal immersion  whose mean curvature vector field  is given by  formula (\ref{Hg-RealForms-formula}). The sheaf of local holomorphic lifts of $\varphi$ (which   coincides with the sheaf of local holomorphic lifts which are space-like $\omega^{g}_H$-isotropic immersions) is isomorphic to the sheaf of local sections of the $H$-bundle $Y\times_\varphi G\to Y$ which are parallel with respect to a flat connection $B_\varphi$ on this bundle; in particular, if $Y$ is simply connected, then  $\varphi$ admits a holomorphic lift $f:Y\to P$, which is unique up to  right translation by an element of $H$, and which is a space-like $\omega^{g}_H$-isotropic immersion.

\end{enumerate}
\end{thry}

\subsection{Examples}
\label{Examples-subsection}

\begin{ex}
$G=\C^n$,  and $g$ the usual inner product on $H=i\R^n$.	 The associated symmetric form $\omega^g_H$ on $\C^n$ is   $-\sum_{i=1}^n dz_i \otimes dz_i$. The reductive homogeneous space $(M,g_M)$ can be identified with the flat $\R^n$ and Theorem \ref{main} specialises to the classical holomorphic  representation theorem for minimal conformal immersions $Y\to\R^n$ in terms of isotropic holomorphic immersions $Y\to \C^n$ (see Theorems \ref{Th1-intro}, \ref{Th2-intro} in section \ref{intro}). 
\end{ex}

\begin{ex} (minimal immersions in tori) Let $\Lambda\subset\R^n$ be a lattice and $T_\Lambda\coloneq \R^n/\Lambda$ the associated torus endowed with its canonical flat metric.  We identify $T_\Lambda$ with the quotient $\C^n/\Lambda+ i\R^n$, noting that $ \Lambda+ i\R^n$ is obviously a (non-connected) real form of $\C^n$ in our sense. 

Let $(Y,J)$ be a connected Riemann surface,  let $\varphi:Y\to T_\Lambda$ be a minimal conformal immersion, and let 
$$
 \begin{tikzcd}
 \tilde Y \ar[r, "\tilde f"] \ar[d, "c"'] & \C^n\ar[d, "\pi_{\Lambda+i\R^n}"]\\
 Y\ar[r, "\varphi"]& T_\Lambda=\qmod{\R^n}{\Lambda	+i\R^n}
 \end{tikzcd},  
 $$
 be a Weierstrass representation, where $c$ is a normal covering map   of Riemann surfaces and $\tilde f$ is a holomorphic immersion which is equivariant with respect to a group morphism $\Aut_Y(\tilde Y)\to \Lambda+i\R^n$ (see Proposition \ref{existence-Weierstrass} (1)). The form $d\tilde f=(d\tilde f_1,\dots,d\tilde f_n)$ obviously descends to $Y$ defining a  $\C^n$-valued holomorphic form $\beta=(\beta_1,\dots,\beta_n)$ on $Y$ with the property
 \begin{equation}\label{cond1-omega}
 \int_\gamma \beta\in 	\Lambda+i\R^n \hbox{ for any smooth loop $\gamma$ in $Y$.}
 \end{equation}
 Since $\tilde f$ is $\sum_i dz_i\otimes dz_i$-isotropic, we have 
\begin{equation}\label{cond2-omega}
\sum_{j=1}^n \beta_j^2=0.	
\end{equation}
Now note that, for a smooth path $\nu:[a,b]\to Y$ in $Y$, we have 
 \begin{equation}
 \varphi(\nu(b))-\varphi(\nu(a))	=\int_\nu \beta + (\Lambda	+i\R^n),
 \end{equation}
which shows that  $\beta$ determines  $f$ up to translations in $T_\Lambda$. Theorem \ref{main-real-forms}  gives the following Weierstrass representation theorem for minimal surfaces in tori (compare with \cite[Theorem 2.1]{ArMi}):
\begin{pr}
The assignment $\varphi\mapsto \beta$ defines 	a bijection between the set of equivalence classes of minimal conformal immersions $\varphi:Y\to T_\Lambda$ (modulo translations in $T_\Lambda$) and the set of nowhere vanishing forms $\beta\in \Omega^1(Y,\C^n)$ satisfying conditions (\ref{cond1-omega}), (\ref{cond2-omega}).  Fixing $y_0\in Y$, the minimal immersion associated with such a form $\beta$ is given by 
$$
\varphi(y)=\tau_0+\bigg(\int_{\nu_y} \beta\bigg) + (\Lambda	+i\R^n), 
$$
where $\tau_0\in T_\Lambda$ and $\nu_y:[0,1]\to Y$ is a smooth path with $\nu_y(0)=y_0$, $\nu_y(1)=y$.

\end{pr}

\end{ex}

\begin{ex}\label{simple-group-ex}
Let $G$ be a  simple complex Lie group, $B$ be the Killing form of $\g$, and $\omega$ be the bi-invariant symmetric holomorphic 2-form on $G$ determined by the condition $\omega_e=B$. Let $H$ be a real form of $G$. The restriction $B_\hg$ of  $B$  to $\hg$ coincides with the Killing form of $\hg$ \cite[Remark 1, p. 316]{SaWa}; this restriction  is $\R$-valued and non-degenerate.

Let $H$ be a real form  of $G$ and  $c\in \R^*$. Endow  the Lie algebra $\hg$ with the non-degenerate symmetric bilinear form $g_H(c)\coloneq c B_\hg$. With this choice we have  $\omega^{g_H(c)}_H=c \omega$, so  the  isotropic-holomorphy conditions  associated with all  pairs of the form $(H, c B_\hg)$  coincide.  

Consider for example the  special case $G=\SL(2,\C)$. This group admits  up two conjugacy two real forms: the compact real form $\SU(2)$, and the split real form $\SL(2,\R)$ (which is equivalent to  $\SU(1,1)$). Choosing  $c=-\frac{1}{2}$ for $\SU(2)$,  and $c=\frac{1}{2}$ for $\SL(2,\R)$ we obtain a model for the hyperbolic space $\H^3$, respectively a model for the de Sitter space $\mathrm{dS}_3$.

This shows that Bryant's theorem (see Theorem \ref{Th3-intro}) and the correspondences  proved in    \cite{AiAk}, \cite{Fu} are special cases of  Theorem \ref{main-real-forms} and Corollary \ref{scalar-mean-curv-coro}.\\

\end{ex}

\begin{ex}
Let $G$ be a connected reductive complex group, $K\subset G$ be maximal compact subgroup,  and  $g$ an $\Ad_K$-invariant inner product its Lie algebra $\kg$. The exponential map gives a diffeomorphism $i\kg\to G/K$, so in this case the reductive homogeneous space $M$ is contractible. The induced metric $g_M$ on this quotient is Riemannian, complete, and has non-positive sectional curvature.
\end{ex}

\begin{ex}\label{Example-reductive}
Let $G$ be a connected complex reductive group, $\theta$  a Weyl involution, and $\tau$    a Cartan involution on $G$ which commutes with $\theta$ \cite[section 2]{Ak}. Let   $K\coloneq G^\tau$ ($K'\coloneq G^{\theta\tau}$) be the fixed point locus of $\tau$ (respectively $\theta\tau$), which is a $\theta$-invariant maximal compact subgroup of $G$  (respectively a split real form of $G$). We denote by $\theta_*$ ($\tau_*$) the induced $\C$-linear (anti-linear) involution on $\g$.  The Lie algebra $\kg'$ decomposes as $\kg_1^{\theta_*}\oplus i\kg_{-1}^{\theta_*}$, where on the right we used the notation $\kg^{\theta_*}_\lambda\coloneq\ker(\theta_*-\lambda\id_\kg)$. Choose an inner product $g$ on $\kg$ which is $\Ad_K$-invariant, and $\theta$-invariant. Let $g^\C:\g\times\g\to \C$ be its $\C$-bilinear extension. The restriction $g'$ of $g^\C$ to $\kg'\times\kg'$ is an $\R$-valued, non-degenerate symmetric bilinear form, which is positive definite on  $\kg_1$ and negative definite on  $i\kg_{-1}$. The signature of $g'$ is 
$$\dim_\R(\kg_1^{\theta_*})-\dim_\R(\kg_{-1}^{\theta_*})=\dim_\C(\g_1^{\theta_*})-\dim_\C(\g_{-1}^{\theta_*})=-\rk(G).$$
The second equality follows using the explicit formulae for $\theta_*$ in a canonical system of generators \cite[p. 18]{On}.
Since $g'$ is the restriction of $g^\C$ to the real form $\kg'$ of $\g$, we have $g'^\C=g^\C$. In others words one has $\omega^g_{K'}=\omega^{g'}_K$, 
so  the  isotropy conditions  associated with the pairs $(K,g)$, $(K',g')$ coincide. We obtain the diagram

\begin{equation}
\begin{tikzcd}[column sep=4mm]
&G\ar[dl, bend right =30, "\pi_K"'] \ar[dr, bend left =30, "\pi_{K'}"]&\\
M\coloneq G/K & &	M'\coloneq G/K'
\end{tikzcd}	
\end{equation}
in which $G$ is endowed with the holomorphic symmetric form $\omega^g_K=\omega^{g'}_{K'}$, the left hand quotient $M$ with the Riemannian metric $g_M$, and  the right hand quotient $M'$ with the pseudo-Riemannian metric $g'_{M'}$.

A standard example of this type is obtained in the case $G=\GL(n,\C)$ by taking 
$$S\stackrel{\theta}{\mapsto} (\trp{\hspace{-0.6mm} S})^{-1},\ S\stackrel{\tau}{\mapsto}(S^*)^{-1},$$
which gives $K=\U(n)$, $K'=\GL(n,\R)$. Taking $g(a,b)=-\Tr(ab)$ for $a$, $b\in\u(n)$, the forms $g^\C$ and $g'$ will be given by the same formula written for matrices in $\gl(n,\C)$, $\gl(n,\R)$ respectively.

The quotient $\GL(n,\C)/\U(n)$ can be identified with the space of Hermitian inner products on $\C^n$, or, equivalently with the space $\Herm_+(n)$ of positive Hermitian matrices. Via this identification, the bundle map $\GL(n,\C)\to \GL(n,\C)/\U(n)$ becomes
$$\pi:\GL(n,\C)\to \Herm_+(n),\ \pi(S)=SS^*.
$$
A positive Hermitian matrix $h\in \Herm_+(n)$ defines a Hermitian inner product $\langle \cdot,\cdot\rangle_h$ given by
$$
\langle u,v\rangle_h= \trp{\bar u}hv=\sum_{i,j}\bar u_i h_{ij}v_j=\langle u, hv\rangle. 
$$
It is well known that the tangent space $T_h(\Herm_+(n))=\Herm(n)$ at any $h\in \Herm_+(n)$ can be identified with the space $\Herm_h\subset\gl(n,\C)$  of matrices which are Hermitian with respect to $\langle \cdot,\cdot\rangle_h$ via the isomorphism
$$T_h(\Herm_+(n))=\Herm(n)\textmap{\iota_h\simeq} \Herm_h,\ u\mapsto h^{-1}u.
$$

A simple computation shows that, via these identifications, the Lie bracket $[\cdot,\cdot]_{\U(n)}$ on the tangent bundle $T_{\Herm_+(n)}$ is given by the formula
\begin{equation}\label{[]U(n)}
\forall (\chi,\chi') \in\Herm_h\times\Herm_h,\  [\chi,\chi']_{\U(n)}=\frac{i}{2}[\chi,\chi'].
\end{equation}

Similarly, the quotient $\GL(n,\C)/K'=\GL(n,\C)/\GL(n,\R)$ can be identified with the space  $R(n)$  of real structures on $\C^n$, i.e.  with the space of anti-linear involutions of $\C^n$. Via this identification the bundle map $\GL(n,\C)\to \GL(n,\C)/\GL(n,\R)$  becomes
$$
\pi':\GL(n,\C)\to R(n),\  \pi'(S)=S\circ c\circ S^{-1},
$$
where $c$ denotes the standard conjugation. The tangent space $T_\sigma R(n)$   at  an anti-linear involution $\sigma\in R(n)$ coincides with the vector space of anti-linear endomorphisms of $\C^n$ which anti-commute with $\sigma$, i.e. which are pure imaginary with respect to $\sigma$. A simple computation shows that, via this identification, the Lie bracket $[\cdot,\cdot]_{\GL(n,\R)}$ on the tangent bundle $T_{R(n)}$ is given by the formula
\begin{equation}\label{[]GL(n,R)}
\forall (\chi,\chi') \in T_\sigma R(n)\times T_\sigma R(n),\  [\chi,\chi']_{\GL(n,\R)}=-\frac{i}{2}[\chi,\chi']\sigma.
\end{equation}

Using formulae (\ref{[]U(n)}), (\ref{[]GL(n,R)}) and Theorem \ref{main-real-forms} (1)(b), one obtains explicit formulae for the mean curvature vector field of the conformal immersions which intervene in the Bryant type correspondences associated with the real forms $\U(n)$, $\GL(n,\R)$ of $\GL(n,\C)$ and the symmetric forms $g$, respectively $g'$. In both cases the isotropy condition is  associated with the bi-invariant symmetric form on $G$ which coincides with $(a,b)\mapsto -\Tr(ab)$ at the unit element $I_n\in\GL(n,\C)$.

\end{ex}

\begin{ex} (An infinite dimensional Bryant correspondence)

Let $X$ be a compact differentiable manifold and $E$ a complex vector bundle on $X$. 
The complex gauge group ${\cal G}^E\coloneq \Aut(E)$ is open in the complex vector space $A^0(X,\End(E))$, and, after suitable Sobolev completions,  it becomes an infinite dimensional complex Lie group (see \cite{DK}, \cite{Te}). Let $h_0$ be a Hermitian metric on $E$ and let ${\cal G}^E_{h_0}\coloneq \Aut(E,h)$ be (a suitable Sobolev completion of) the  gauge group of unitary automorphisms of  the Hermitian bundle $(E,h_0)$. Denoting by $\GL(E)$, $\U(E,h_0)$ the group bundles of fibrewise linear (respectively unitary) isomorphisms, we have ${\cal G}^E=\Gamma(X,\GL(E))$, ${\cal G}^E_{h_0}=\Gamma(X,\U(E,h_0))$. Omitting again Sobolev completions,  the Lie algebras of the two gauge groups are 
$$
\mathrm{Lie}({\cal G}^E)=A^0(X,\End(E)), \ \mathrm{Lie}({\cal G}^E_{h_0})=A^0(X, \u(E,h_0)),
$$
where $\u(E,h_0)$ denotes the real vector bundle of anti-Hermitian (with respect to $h_0$) endomorphisms. This shows that ${\cal G}^E_{h_0}$ is a real form of ${\cal G}^E$ in the sense explained in section \ref{RealForms-subsection} (extended in the obvious way to infinite dimensional Lie groups). 

Now fix an orientation and a Riemannian metric $g$ on $X$ and denote by $\vol_g$ the corresponding volume form. The formula
$$
\langle a, b\rangle_{L^2}\coloneq -\int_X \Tr (ab)\vol_g
$$
defines an $\Ad$-invariant inner product on $\mathrm{Lie}({\cal G}^E_{h_0})=A^0(X, \u(E,h_0))$.

It is well known that the quotient ${\cal M}^E\coloneq {\cal G}^E/{\cal G}^E_{h_0}$ can be identified with the space of Hermitian metrics on $E$, and, via this identification, the bundle map $\pi: {\cal G}^E\to {\cal M}^E$ is given by the formula
$$
\pi(S)(u,v)=h_0(S^*(u),S^*(v))=h_0(u,SS^*(v)),
$$
where $(-)^*$ denotes adjoint with respect to  $h_0$. For a Hermitian metric $h\in {\cal M}^E$, the tangent space $T_h{\cal M}^E$ can be identified with the space $A^0(X,\Herm(E,h))$ of endomorphisms of $E$ which are Hermitian (self-adjoint) with respect to $h$. Using these identifications, we see that the Lie algebra structure $[\cdot,\cdot]_{{\cal G}^E_{h_0}}$ on a  tangent space $T_h{\cal M}=A^0(X,\Herm(E,h))$ obtained using the general formalism of section \ref{RealForms-subsection} is given by the known formula (\ref {[]U(n)}) written for global endomorphisms $\chi$, $\chi'\in A^0(X,\Herm(E,h))$.

The holomorphic symmetric bilinear form $\omega^E$ on the complex Lie group ${\cal G}^E$  which is associated with $\langle \cdot , \cdot \rangle_{L^2}$  following the general formalism developed in \ref{RealForms-subsection} is given by
\begin{equation} 
\omega^E_S(u,v)=-\int_X \Tr(S^{-1}u\; S^{-1}v)\vol_g
\end{equation}
for any $S\in {\cal G}^E$ and $u$, $v\in T_S{\cal G}^E=A^0(X,\End(E))$.

\begin{re} Let $Y$ be a Riemann surface. The general Theorem \ref{main-real-forms} is also valid for the triple 
$$
({\cal G}^E, {\cal G}^E_{h_0}, \langle\cdot,\cdot\rangle_{L^2})
$$
and gives a Bryant type correspondence between 
\begin{enumerate} 
\item Holomorphic immersions $f:Y\to 	{\cal G}^E$ which are isotropic with respect to $\omega^E$.
\item Conformal immersions $\varphi: Y\to {\cal M}^E$ whose mean curvature vector field is given by  the formula $\vec{\Hg}_\varphi=\varphi^*([\cdot,\cdot]_{{\cal G}^E_{h_0}})/\vol_\varphi$, where $[\cdot,\cdot]_{{\cal G}^E_{h_0}}$ is the Lie algebra structure on the tangent bundle $T_{{\cal M}^E}$ given by the formula
$$
[\chi,\chi']_{{\cal G}^E_{h_0}}=\frac{i}{2}[\chi,\chi']
$$
for any $h\in {\cal M}^E$ and $\chi$, $\chi'\in T_h{\cal M}^E=A^0(X,\Herm(E,h))$.
\end{enumerate}

\end{re}

\end{ex}

\subsection{The Bryant correspondence for  bundles of the form \texorpdfstring{$\Gamma\backslash G\to \Gamma\backslash G/H$}{Gamma}}\label{Bryant-for-Gamma-section}

Let $G$ be a complex Lie group and $H$ a real form of $G$ in the sense of section \ref{RealForms-subsection} and let $g$ be a $\Ad_{H}$-invariant,   symmetric, $\R$-bilinear form on its Lie algebra $\hg$.  Let $\Gamma\subset G$ be a discrete subgroup of $G$ such that the natural left $\Gamma$-action on 	$M=G/H$ is
\begin{enumerate}
\item[(PD)]  properly discontinuous,
\item[(FR)] free. 
\end{enumerate}
\begin{re}
By a well known remark of Thurston, the first condition is automatically fulfilled if $H$ is compact (see Corollary \cite[Corollary 3.5.11 p. 157]{Th}). The second condition is equivalent to: $\Gamma\bigcap ( \underset{{g\in G}}{\bigcup}gHg^{-1})=\{e\}$.
\end{re} 
Let $$
\pi_\Gamma: P_\Gamma\coloneq \Gamma\backslash G\to \Gamma\backslash M\eqcolon M_\Gamma 
$$
the principal $H$-bundle obtained from $\pi_{H}$ by taking   $\Gamma$-quotients. The connection $A$ and the tensorial form  $\alpha$ associated with the real form $H$ of $G$ (see Remark \ref{alpha-A-of-real-form}) descend to $P_\Gamma$ and the obtained couple $(\alpha_\Gamma,A_\Gamma)$ satisfies the hypothesis of Theorem \ref{main}.   Similarly, the bilinear symmetric holomorphic form $\omega^g_{H}$ on $G$ induces a symmetric holomorphic form $\omega_\Gamma$ on $P_\Gamma$ and the pseudo-Riemannian metric $g_{M}$ on $M=G/H$ induces a   pseudo-Riemannian metric $g_{M_\Gamma}$ on $M_\Gamma$. The objects $g_{M_\Gamma}$, $\omega_\Gamma$ are associated with the triple $(\alpha,A,g)$  in the sense of section \ref{g-alpha-M-omega-section}. We also obtain a Lie bracket $[\cdot,\cdot]_H^\Gamma:T_{M_\Gamma}\times_{M_\Gamma}T_{M_\Gamma}\to T_{M_\Gamma}$ induced by $[\cdot,\cdot]_H$.
 
Now let $(Y,J)$ be a Riemann surface.  Our general Theorem \ref{main} applies and gives a Bryant type correspondence between space-like holomorphic $\omega_\Gamma$-isotropic immersions 
$$f:Y\to P_\Gamma=\Gamma\backslash G$$
 and space-like conformal   immersions 
 $$\varphi:Y\to M_\Gamma =\Gamma\backslash M$$
  into the pseudo-Riemannian  manifold $(M_\Gamma,g_{M_\Gamma})$ whose mean curvature vector field is given by formula (\ref{Hg-RealForms-formula}) with $[\cdot,\cdot]_H^\Gamma$ instead of $[\cdot,\cdot]_H$ on the right. This correspondence becomes of course especially interesting when $Y$ and $M_\Gamma$ are compact.
\begin{re}
One can use the theory of orbifolds to generalise this correspondence  for discrete subgroups  $\Gamma\subset G$ satisfying only condition (PD). This condition is sufficient for defining a natural orbifold structure on the quotient $M_\Gamma$ (see \cite[Proposition 13.2.1 p. 302]{Th0}). One has to generalise the theory of connections and tensorial forms (see section \ref{connections-section}) to include principal bundles over orbifolds and also the theory of immersions (see section \ref{immersions-section}) to include the case when the target manifold is a pseudo-riemannian orbifold. 
\end{re}

\begin{ex} 
Coming back to Example  \ref{simple-group-ex}, consider the principal  bundle  
$$\pi_{\SU(2)}:\SL(2,\C)\to  \SL(2,\C)/\SU(2)=\H^3 $$
associated with the real form $\SU(2)$ of $\SL(2,\C)$. The Riemannian metric $g_{\H^3}$ on $\H^3$ associated with the inner product 
$$(a,b)\stackrel{g}{\mapsto}g(a,b)=-2\Tr(a b)=-\frac{1}{2}B_{\su(2)}(a,b)$$
on $\su(2)$ coincides the canonical hyperbolic metric on  $\H^3$. 

 Let $\Gamma_0\subset \PSL(2,\C)$ be a torsion free Kleinian group,   $\Gamma\subset \SL(2,\C)$  a lift of $\Gamma_0$ in $\SL(2,\C)$ (see \cite{Cu}). Let $M_\Gamma\coloneq \Gamma\backslash \H^3=\Gamma_0\backslash \H^3$ the associated hyperbolic 3-manifold, and 
$$
\pi_\Gamma: P_\Gamma\coloneq \Gamma\backslash\SL(2,\C)\to \Gamma\backslash \SL(2,\C)/\SU(2)= M_\Gamma 
$$
the principal $\SU(2)$-bundle obtained from $\pi_{\SU(2)}$ by taking   $\Gamma$-quotients.

 Theorem \ref{main} and Corollary \ref{scalar-mean-curv-coro} apply and give a correspondence between holomorphic $\omega_\Gamma$-isotropic immersions $f:Y\to P_\Gamma=\Gamma\backslash \SL(2,\C)$ and conformal CMC-1  immersions $\varphi:Y\to M_\Gamma =\Gamma\backslash \SL(2,\C)/\SU(2)$ into the {\it oriented} hyperbolic manifold $(M_\Gamma,g_{M_\Gamma})$.
 
 The existence of a simple Weierstrass representation of a conformal CMC-1  immersion  $\varphi:Y\to M_\Gamma$ (see section \ref{Weierstrass-repr-section}) is especially interesting when both $Y$ and $M_\Gamma$ are compact. Indeed, the existence of closed holomorphic curves of a specified genus in a  compact quotient of the form $\Gamma\backslash\SL(2,\C)$ is an old, classical problem in complex geometry. 

Winkelmann proved  that any such compact quotient admits genus 1 holomorphic curves (see \cite[Proposition 4.3.2 p. 60]{Wi}) and asked if there exists such quotients which admit holomorphic curves of genus $g\geq 2$ \cite[p. 58]{Wi}). The main result of \cite{BDHH} solves this longstanding problem by giving a positive answer to Winkelmann's question. 

Taking into account this result, the authors come to the natural question: can one obtain a holomorphic embeddings $Y\hookrightarrow  P_\Gamma $ of a compact Riemann surface $Y$ in a compact quotient $P_\Gamma=\Gamma\backslash\SL(2,\C)$    as a simple Weierstrass representation $f:Y\to P_\Gamma$ of a CMC-1 embedding $\varphi:Y\hookrightarrow M_\Gamma$? If yes, the obtained embedding will be not only holomorphic, but also $\omega_\Gamma$-isotropic.

  This gives a strong motivation for studying the existence of conformal CMC-1 embeddings (and, more generally, CMC-1  immersions) admitting a simple Weierstrass representation.
  
  \end{ex}

\end{document}